\documentclass{amsart}
\usepackage{amscd}
\usepackage{amssymb, url, hyperref}
\usepackage[all]{xy}

\newtheorem{thm}{Theorem}[section]
\newtheorem{lem}[thm]{Lemma}
\newtheorem{cor}[thm]{Corollary}
\newtheorem{pro}[thm]{Proposition}

\theoremstyle{definition}

\DeclareMathOperator{\Card}{Card}
\DeclareMathOperator{\Int}{Int}

\DeclareMathOperator{\F}{F}

\DeclareMathOperator{\Hom}{Hom}

\DeclareMathOperator{\Ad}{Ad}
\DeclareMathOperator{\Gr}{Gr}

\DeclareMathOperator{\rank}{rank}
\DeclareMathOperator{\ch}{ch}

\begin{document}
\title[Discrete Series]{On the geometric approach to the discrete series} 
\author[D. Mili\v ci\'c]{Dragan Mili\v ci\'c}
\address{(Mili\v ci\'c) Department of Mathematics, University of Utah, 
Salt Lake City, UT 84112, USA}
\email{milicic@math.utah.edu}
\urladdr{http://www.math.utah.edu/\textasciitilde milicic}
\author[A. Romanov]{Anna Romanov}
\address{(Romanov) School of Mathematics and Statistics, University of New South Wales, Sydney,
  NSW 2052, Australia}
\email{a.romanov@unsw.edu.au}
\urladdr{http://web.maths.unsw.edu.au/\textasciitilde aromanov/}

\dedicatory{Dedicated to the memory of Harish-Chandra, in admiration.}

\maketitle

\section{Introduction}
Let $G_0$ be a connected semisimple Lie group with finite
center. Discrete series representations of $G_0$ are the irreducible
unitary representations of $G_0$ with square integrable matrix
coefficients. They appear discretely in the decomposition of the
regular representation of $G_0$ on $L^2(G_0)$; i.e., their Plancherel
measure is positive. In his two seminal papers \cite{ds1} and
\cite{ds2}, Harish-Chandra determined the necessary and sufficient
condition for $G_0$ to admit discrete series representations. Let
$K_0$ be a maximal compact subgroup of $G_0$. Then $G_0$ has
discrete series if and only if the ranks of $G_0$ and $K_0$ are
equal. In this situation, a maximal torus $T_0$ in $K_0$ is a
(compact) Cartan subgroup in $G_0$. A discrete series representation
has to have a regular real infinitesimal character. Moreover, the
(distribution) character of a discrete series representation must be a
tempered invariant eigendistribution on $G_0$.

In the first paper \cite{ds1}, Harish-Chandra establishes that every
tempered invariant eigendistribution with regular real infinitesimal
character is completely determined by its restriction to the elliptic
set; i.e., the set of conjugacy classes represented by regular
elements of $T_0$. He also describes all of these restrictions
explicitly. While the first statement is a straightforward application
of Harish-Chandra's ``matching conditions'', the exhaustion part in the
second statement is technically difficult.

In the second paper \cite{ds2}, Harish-Chandra relates the above
invariant eigendistributions to the characters of discrete series
representations, by giving explicit formulas for the restriction of
the latter to the elliptic set in $G_0$. In particular, the characters of
discrete series span the space of tempered invariant
eigendistributions with regular real infinitesimal character. This
part of the proof is intertwined with the proof of the ``discrete
part'' of the Plancherel formula for $G_0$ and an analogue of the
orthogonality relations for discrete series characters (compare
\cite{kazhdan}). The flowchart of the argument is conceptually very
similar to Hermann Weyl's proof of the character formula for
irreducible finite-dimensional representations of connected compact
Lie groups (though the adaptation to the discrete series requires new
ideas and introduces tremendous technical complications).

In \cite{hmsw2}, a different approach to the discrete series is
introduced. It is based on the localization theory of Harish-Chandra
modules developed by Alexander Beilinson and Joseph Bernstein
\cite{bebe1}, \cite{bebe2}. Let $\mathfrak g$ be the complexified Lie
algebra of $G_0$ and $X$ the flag variety of $\mathfrak g$. Let $K$ be
the complexification of $K_0$. Under certain positivity conditions,
they establish the equivalence of categories of Harish-Chandra sheaves
(i.e., $K$-equivariant coherent $\mathcal D$-modules) on $X$ with the
categories of Harish-Chandra modules of the pair $(\mathfrak
g,K)$. Using the connection of $\mathfrak n$-homology (for $\mathfrak
n$ corresponding to the nilpotent subgroup in Iwasawa decomposition of
$G_0$) with leading exponents in the Harish-Chandra expansion of
matrix coefficients \cite[II.2, Thm.~1]{mil}, \cite[8.24]{cm}, one
establishes the relationship between the support of the localizations of
an irreducible Harish-Chandra module and asymptotics of its matrix
coefficients. This result gives a geometric proof of the existence
criterion for discrete series. Moreover, it establishes that the
discrete series are global sections of (irreducible) standard
Harish-Chandra sheaves $\mathcal I(Q,\tau)$ attached to closed
$K$-orbits $Q$ in the flag variety $X$ and irreducible
$K$-equivariant connections $\tau$ on $Q$ compatible with a real
regular infinitesimal character. A simple counting argument shows that
Harish-Chandra's list of irreducible characters of discrete series
agrees with the geometric construction via $\mathcal D$-modules, but
the explicit connection is not immediately clear.

In this paper we make this connection explicit. Our argument is based
on a simple geometric formula for the $\mathfrak n$-homology of
modules over the enveloping algebra $\mathcal U(\mathfrak g)$ for a
regular infinitesimal character proved in Theorem \ref{formula}. Let
$x$ be a point in the flag variety, $\mathfrak b_x$ the corresponding
Borel subalgebra of $\mathfrak g$ and $\mathfrak n_x = [\mathfrak
  b_x,\mathfrak b_x]$. Let $N_x$ be the unipotent subgroup
corresponding to $\mathfrak n_x$; its orbits in $X$ define a Bruhat
stratification of $X$. Roughly speaking, the weight subspaces of
$\mathfrak n_x$-homology of a Harish-Chandra module are determined by
the (derived) direct images to a point of restrictions of the
corresponding Harish-Chandra sheaf to each Bruhat cell.  In Section
\ref{nhomology}, we use this formula to reprove a result of Wilfried
Schmid which describes the $\mathfrak n$-homology of the discrete
series representations $\Gamma(X,\mathcal I(Q,\tau))$ with respect to
nilpotent radical $\mathfrak n$ of a Borel subalgebra $\mathfrak b$
which is stable under the action of the Cartan involution attached to
$K_0$. Using a special case of the Osborne conjecture \cite{osborne},
this gives the formula for the character of these representations on
the elliptic set in $G_0$. This proof is formally similar to the
deduction of the Weyl character formula from Kostant's formula for
$\mathfrak n$-homology of irreducible finite-dimensional
representations of connected compact Lie groups. Establishing the
formula for the character of $\Gamma(X,\mathcal I(Q,\tau))$ makes the
map from geometric parameters of discrete series into Harish-Chandra
parameters completely explicit. As a consequence it also implies that
the space of all tempered invariant eigendistributions on $G_0$ with a
given regular infinitesimal character is spanned by the characters of
discrete series; what immediately implies the most difficult technical
result in \cite{ds1} (the existence of tempered invariant
eigendistributions $\Theta_\lambda$ in \cite[Thm.~3]{ds1}).

The machinery established above allows us to prove some other results
on the discrete series by geometric methods. As an illustration, we
explain in Section \ref{blattner_conj} how a natural $K$-equivariant
filtration of standard Harish-Chandra sheaves $\mathcal I(Q,\tau)$
leads to a very simple proof of Blattner's conjecture \cite{blattner}.

Our arguments can be extended to the Harish-Chandra class of reductive
Lie groups as explained in an appendix to \cite{hmsw1} and
\cite{ha1}. To simplify the notation and reduce a number of technical
issues, we leave this as an exercise to an interested reader.

During two weeks in October 2023, one of us (D.~M.) visited at
Harish-Chandra Research Institute in Prayagraj, UP, India, during
Centennial Celebration of the birth of Harish-Chandra. In the first
week, during the workshop “Representation theory of real Lie groups
and automorphic forms”, he gave a series of lectures on the modern
view of basic results of Harish-Chandra on representation theory of
reductive Lie groups. During the lectures and in inspiring discussions
after them, the audience asked a lot of questions about the relation
between the original results of Harish-Chandra and their “modern”
interpretation. The present paper is an attempt to answer some of
these questions related to discrete series in a coherent fashion. We
thank the organizers for their hospitality and inspiring atmosphere
during the workshop.

The first draft of the paper was written while the authors were
visiting Sydney Mathematical Research Institute at the University of
Sydney. We thank them for their hospitality and stimulating
atmosphere.

\section{A geometric formula for $\mathfrak n$-homology}
\label{geo_formula}
\subsection{Geometric preliminaries}
We start with recalling some basic geometric setup (one can consult
\cite[Ch.~II]{book} for more details). Let $\mathfrak g$ be a complex
semisimple Lie algebra and $X$ its flag variety. For any $x \in X$, we
denote by $\mathfrak b_x$ the corresponding Borel subalgebra. Denote
by $\mathcal B$ the tautological vector subbundle of the trivial
bundle $X \times \mathfrak g$ over $X$ with fiber $\mathfrak b_x$ over
$x \in X$. For any $x \in X$, we denote by $\mathfrak n_x = [\mathfrak
  b_x,\mathfrak b_x]$. Let $\mathcal N$ be the vector subbundle of
$\mathcal B$ with fibers $\mathfrak n_x$, $x \in X$. Then the quotient
vector bundle $\mathcal H = \mathcal B/\mathcal N$ is trivial. We
denote by $\mathfrak h$ the space of its global sections. We call
$\mathfrak h$ the {\em (abstract) Cartan algebra} of $\mathfrak
g$. Let $\mathfrak c$ be a Cartan subalgebra of $\mathfrak g$
contained in $\mathfrak b_x$. Then there is a canonical linear
isomorphism of $\mathfrak c$ with $\mathfrak h$. The dual isomorphism
of $\mathfrak h^*$ with $\mathfrak c^*$ we call the {\em
  specialization} at $x \in X$.

Let $\mathcal U(\mathfrak g)$ be the enveloping algebra of $\mathfrak
g$ and $\mathcal Z(\mathfrak g)$ its center. We have the canonical
Harish-Chandra homomorphism $\gamma : \mathcal Z(\mathfrak g)
\longrightarrow \mathcal U(\mathfrak h)$ \cite[Ch.~VII, \S
  6, no.~4]{bour}, defined in the following way. For any $x \in X$,
$\mathcal Z(\mathfrak g)$ is contained in the sum of the subalgebra
$\mathcal U(\mathfrak b_x)$ and the right ideal $\mathfrak n_x
\mathcal U(\mathfrak g)$ of $\mathcal U(\mathfrak g)$. Hence, we have
the natural projection of $\mathcal Z(\mathfrak g)$ into $\mathcal
U(\mathfrak b_x )/(\mathfrak n_x U(\mathfrak g) \cap \mathcal
U(\mathfrak b_x)) = \mathcal U(\mathfrak b_x)/\mathfrak n_x \mathcal
U(\mathfrak b_x ) = \mathcal U(\mathfrak c )$.  Its composition with
the natural isomorphism of $\mathcal U(\mathfrak c )$ with $\mathcal
U(\mathfrak h)$ is independent of $x$ and, by definition, equal to
$\gamma$.

The specialization $\mathfrak h^* \longrightarrow \mathfrak c^*$
identifies the roots of the pair $(\mathfrak g,\mathfrak c)$ with a
reduced root system $\Sigma$ in $\mathfrak h^*$ which we call the {\em
  root system} of $\mathfrak g$. We choose a positive set of roots
$\Sigma^+$ in $\Sigma$ such that, for any $x \in X$, the root
subspaces corresponding to the specialization of these roots span
$\mathfrak n_x$. Let $\Pi$ be the set of simple roots determined by
$\Sigma^+$. Denote by $W$ the Weyl group of $\Sigma$ and $S$ the set
of reflections with respect to simple roots in $\Pi$. We denote by
$\ell : W \longrightarrow \mathbb Z_+$ the corresponding length
function on $W$. Let $\rho$ be the half sum of roots in $\Sigma^+$.

Since $\mathfrak h$ is abelian, the enveloping algebra $\mathcal
U(\mathfrak h)$ is equal to the symmetric algebra $S(\mathfrak h)$ of
$\mathfrak h$, which is isomorphic to the algebra of all polynomials
on $\mathfrak h^*$. Therefore, for any $\lambda \in \mathfrak h^*$,
the composition $\varphi_\lambda$ of the homomorphism $\gamma :
\mathcal Z(\mathfrak g) \longrightarrow \mathcal U(\mathfrak h)$ with
the evaluation at $\lambda+\rho$ is a character of $\mathcal
Z(\mathfrak g)$. By a classic result of Harish-Chandra \cite[Ch.~VIII,
  \S 2, no.~5, Cor.~1 of Thm.~2]{bour}, $\ker \varphi_\lambda = \ker
\varphi_\mu$ if and only if there exists $w \in W$ such that $\lambda
= w \mu$. This map establishes a bijection of $W$-orbits in $\mathfrak
h^*$ and maximal ideals in $\mathcal Z(\mathfrak g)$. Let $\theta = W
\lambda \subset \mathfrak h^*$ be a $W$-orbit, and denote by
$J_\theta = \ker \varphi_\lambda$ the corresponding maximal
ideal in $\mathcal Z(\mathfrak g)$. We denote by $\mathcal U_\theta$
the quotient $\mathcal U(\mathfrak g)/J_\theta \mathcal U(\mathfrak
g)$.

\subsection{Localization of $\mathcal U_\theta$-modules}
For any $\lambda \in \mathfrak h^*$, Beilinson and Bernstein \cite{bebe1}
constructed a twisted sheaf of differential operators $\mathcal
D_\lambda$ on the flag variety $X$ with a natural algebra homomorphism
$\mathcal U(\mathfrak g) \longrightarrow \Gamma(X,\mathcal
D_\lambda)$.  Moreover, this homomorphism factors through $\mathcal
U_\theta$ and the induced map $\mathcal U_\theta \longrightarrow
\Gamma(X,\,\mathcal D_\lambda)$ is an isomorphism \cite[Ch.~II, Thm.~6.1.(i)]{book}.

Denote by $\mathcal M(\mathcal U_\theta)$ the category of $\mathcal
U_\theta$-modules, and by $\mathcal M(\mathcal D_\lambda)$ the
category of (quasicoherent) $\mathcal D_\lambda$-modules on $X$. One define
the functors
$$
\xymatrix{
\mathcal M(\mathcal D_\lambda) \ar@/^/[rr]^{\Gamma(X,-)} & & \mathcal M(\mathcal 
U_\theta) \ar@/^/[ll]^{\Delta_\lambda}
}
$$
where $\Gamma(X,-)$ is the functor of global sections and
$$
\Delta_\lambda(V) = \mathcal D_\lambda \otimes_{\mathcal U_\theta} V
$$ 
for a module $V$ in $\mathcal M(\mathcal U_\theta)$.
The functor $\Delta_\lambda$ is called the {\em localization at
  $\lambda$}. Clearly, the functor $\Delta_\lambda$ is a left adjoint
of $\Gamma(X,-)$.

Let $\Sigma\, \check{}\,$ be the dual root system of $\Sigma$. Denote by
$\alpha\, \check{}\,$ the dual root of $\alpha$. We say that $\lambda
\in \mathfrak h^*$ is {\em regular} if $\alpha\, \check{}\, (\lambda)
\ne 0$ for any $\alpha \in \Sigma$.

Assume that the orbit $\theta$ consists of regular elements of
$\mathfrak h^*$. Then $\mathcal U_\theta$ has finite cohomological
dimension \cite[Ch.~III, Thm.~1.4]{book}. Therefore there exists the
  left derived functor $L \Delta_\lambda$ between bounded derived
  categories $D^b(\mathcal D_\lambda)$ of $\mathcal M(\mathcal
  D_\lambda)$ and $D^b(\mathcal U_\theta)$ of $\mathcal M(\mathcal
  U_\theta)$. This functor is a left adjoint of the functor $R\Gamma$
  from $D^b(\mathcal D_\lambda)$ into $D^b(\mathcal
  U_\theta)$. Beilinson and Bernstein proved the following result \cite{bebe2}.

\begin{thm}
  \label{equder}
  Let $\theta$ be a regular $W$-orbit in $\mathfrak h^*$ and $\lambda \in \theta$.
  Then the functors
$$
\xymatrix{
D^b(\mathcal D_\lambda) \ar@/^/[rr]^{R\Gamma} & & D^b(\mathcal 
U_\theta) \ar@/^/[ll]^{L \Delta_\lambda}
}
$$
are mutually quasiinverse equivalences of categories.
\end{thm}

We say that $\lambda \in \mathfrak h^*$ is {\em antidominant} if
$\alpha \check{\ }(\lambda) \notin \{1,2, \ldots \}$ for all $\alpha
\in \Sigma^+$. For antidominant $\lambda$, the above theorem is a
direct consequence of the following result \cite{bebe1}.
\begin{thm}
Let $\lambda \in \mathfrak h^*$ be antidominant and regular. 
Then the functors
$$
\xymatrix{
  \mathcal M(\mathcal D_\lambda) \ar@/^/[rr]^{\Gamma(X,-)} & & \mathcal M(\mathcal U_\theta)
  \ar@/^/[ll]^{\Delta_\lambda}
}
$$
are mutually quasi-inverse equivalences of categories.
\end{thm}
One can view this result as a vast generalization of the Borel-Weil theorem.

\subsection{Intertwining functors}
Let $\lambda \in \mathfrak h^*$ be regular. For any $w \in W$, the
functor $L \Delta_{w\lambda} \circ R \Gamma : D^b(\mathcal D_\lambda)
\longrightarrow D^b(\mathcal D_{w\lambda})$ is an equivalence of
categories by Theorem \ref{equder}.

In \cite{bebe2}, Beilinson and Bernstein describe this functor
geometrically under some additional conditions. We sketch their
construction here (compare \cite[Ch.~III, Sec.~3]{book}). Let $Z_w$ be
the variety of pairs of Borel subalgebras in relative position $w$ in
$X \times X$. Then $Z_w$, $w \in W$, are the orbits for the diagonal
action of the group $\Int(\mathfrak g)$ of inner automorphisms of
$\mathfrak g$ on $X \times X$. Denote by $p_1 : Z_w \longrightarrow X$
and $p_2 : Z_w \longrightarrow X$ the restrictions of the projections
to the first, resp.~second factor in $X \times X$.  Then $p_1$ and
$p_2$ are locally trivial fibrations with fibers which are affine
spaces of dimension $\ell(w)$ \cite[Ch.~III, 3.2]{book}. In
particular, $\dim Z_w = \ell(w) + \dim X$ for any $w \in W$.

The twisted sheaf of differential operators $\mathcal D_\lambda$ on
$X$ determines a compatible twisted sheaf $\mathcal D_\lambda^{p_i}$
on $Z_w$ for $i=1,2$. Let $p_2^+ : \mathcal M(\mathcal D_\lambda) \longrightarrow
\mathcal M(\mathcal D_\lambda^{p_2})$ be the $\mathcal D$-module
inverse image functor. This functor is exact and therefore determines
trivially the functor between corresponding derived categories of
$\mathcal D$-modules. Analogously, we have the direct image functor
$p_{1,+} : D^b(\mathcal D_{w\lambda}^{p_1}) \longrightarrow
D^b(\mathcal D_{w\lambda})$. Since the twisted sheaves of
differential operators $\mathcal D_\lambda^{p_2}$ and $\mathcal
D_{w\lambda}^{p_1}$ differ by a twist by the invertible $\mathcal
O_{Z_w}$-module $\mathcal T_w = p_1^*(\mathcal O(\rho - w \rho))$, we
can define the functor
$$
\mathcal V^\cdot \longrightarrow p_{1+}(\mathcal T_w
\otimes_{\mathcal O_{Z_w}} p_2^+(\mathcal V^\cdot))
$$
from $D^b(\mathcal D_\lambda)$ into $D^b(\mathcal D_{w \lambda})$.

There exists a right exact functor $I_w :\mathcal M(\mathcal
D_\lambda) \longrightarrow \mathcal M(\mathcal D_{w\lambda})$ such
that its left derived functor $LI_w : D^b(\mathcal D_\lambda)
\longrightarrow D^b(\mathcal D_{w\lambda})$ is the functor above
$$
LI_w(\mathcal V^\cdot) 
= p_{1+}(\mathcal T_w \otimes_{\mathcal O_{Z_w}} p_2^+(\mathcal V^\cdot))
$$
for any $\mathcal V^\cdot$ in $D^b(\mathcal D_\lambda)$. The functor $LI_w$ is
the {\em intertwining functor} corresponding to $w \in W$.

For any subset $\Theta$ of $\Sigma^+$, we say that $\lambda \in \mathfrak
h^*$ is {\em $\Theta$-antidominant} if $\alpha \check{\ }(\lambda)$ is not
a strictly positive integer for any $\alpha \in \Theta$. Put
$$
\Sigma_w^+ = \{ \alpha \in \Sigma^+ \mid w \alpha \in - \Sigma^+\}
$$
for $w \in W$. The following result is what we alluded to above.

\begin{thm}
Let $w \in W$ and let $ \lambda \in \mathfrak h^*$ be
$\Sigma_w^+$-antidominant and regular. Then the functors $LI_w \circ
L\Delta_\lambda$ and $L\Delta_{w\lambda}$ are isomorphic.
\end{thm}

In particular, we have the following.

\begin{cor}
  \label{locdet}
  Let $w \in W$ and $\lambda \in \mathfrak h^*$ be regular
  antidominant. Then the functors $L\Delta_{w\lambda}$ and $LI_w \circ
  \Delta_\lambda$ are isomorphic.
\end{cor}

Hence, the localization $\Delta_\lambda(V)$ of a $\mathcal
U_\theta$-module $V$, for a regular antidominant $\lambda \in \theta$,
determines all (derived) localizations $L^p \Delta_{w \lambda}(V), \ p
\in \mathbb Z_+$.

\subsection{$\mathfrak n$-homology}
Let $\theta$ be a $W$-orbit consisting of regular elements. Let $V$ be
an object in $\mathcal M(\mathcal U_\theta)$.

Let $x \in X$. As before, we pick a Cartan subalgebra $\mathfrak c$ of
$\mathfrak g$ contained in $\mathfrak b_x$. The Lie algebra
homology groups $H_\cdot(\mathfrak n_x,V)$ are $\mathfrak
c$-modules. Via the dual of the specialization map, we can view them
as $\mathfrak h$-modules.

By a result of Casselman and Osborne \cite{casosb},
\cite{var}, (compare \cite[Ch.~III, 2.4]{book}), we know that the Lie
algebra homology groups $H_\cdot(\mathfrak n_x,V)$ are semisimple
$\mathfrak h$-modules. If we denote by $H_p(\mathfrak n_x,V)_{(\mu)}$
the $\mu$-eigenspace of $H_p(\mathfrak n_x,V)$ for any $\mu \in
\mathfrak h^*$, we have
$$
H_p(\mathfrak n_x,V) = \bigoplus_{w \in W} H_p(\mathfrak n_x,V)_{(w\lambda + \rho)}
$$
for any $p \in \mathbb Z_+$.

Therefore, to calculate Lie algebra homology $H_\cdot(\mathfrak
n_x,V)$ we have to calculate $H_\cdot(\mathfrak n_x,V)_{(w \lambda +
  \rho)}$ for all $w \in W$. In this section, we prove a
formula for this calculation (Theorem \ref{formula} below).

 For an abelian category $\mathcal A$ we denote by $D : \mathcal A
 \longrightarrow D^b(\mathcal A)$ the functor which sends an object
 $A$ of $\mathcal A$ into the complex $D(A)$ which is $A$ in degree
 $0$ and 0 in other degrees.

 Let $\mathcal O_X$ be the sheaf of regular functions on $X$ and
 $\mathcal O_{X,x}$ the stalk of $\mathcal O_X$ at $x \in X$.  Let
 $\mathcal V$ be an $\mathcal O_X$-module. For $x \in X$, we denote by
 $\mathcal V_x$ the stalk of $\mathcal V$ at $x$. Also, we denote by
 $\mathfrak m_x$ the maximal ideal in $\mathcal O_{X,x}$ consisting of
 germs vanishing at $x$. The {\em geometric fiber} $T_x(\mathcal V)$
 of $\mathcal V$ is $\mathcal O_{X,x}/{\mathfrak m_x} \otimes
 _{\mathcal O_{X,x}} \mathcal V_x$.

First, we observe that the above components of $\mathfrak
n_x$-homology are related to derived geometric fibers of derived
localizations. We have the following formula
$$
H_p(\mathfrak n_x, V)_{(\mu + \rho)} = H^{-p}(LT_x(L\Delta_\mu(D(V)))) 
$$ for any $\mu \in \theta$ and $p \in \mathbb Z_+$ \cite[Ch.~III,
  2.6]{book}.  Combining this with Corollary \ref{locdet}, we see that
for regular antidominant $\lambda$,
$$
H_p(\mathfrak n_x, V)_{(w\lambda + \rho)} =
 H^{-p}(LT_x(LI_w(D(\Delta_\lambda(V)))))
$$ 
for any $\mathcal U_\theta$-module $V$. 

Let $i_x : \{x\} \longrightarrow X$ be the canonical inclusion. Then
the $\mathcal D$-module inverse image functor $Li_x^+$ is equal to $LT_x$.

It follows that 
$$
(LT_x \circ LI_w)(\mathcal V^\cdot) 
= (Li_x^+ \circ LI_w)(\mathcal V^\cdot)
= (Li_x^+ \circ p_{1+})(\mathcal T_w \otimes_{\mathcal O_{Z_w}} p_2^+(\mathcal V^\cdot))
$$
for any bounded complex $\mathcal V^\cdot$ in $D^b(\mathcal D_\lambda)$.

Moreover, $Y = p_1^{-1}(\{x\})$ is a closed
$\ell(w)$-dimensional affine subspace in $Z_w$. The codimension of $Y$
in $Z_w$ is $\dim X$.

Denote by $j : Y \longrightarrow Z_w$ the natural inclusion. Let $q_i
: Y \longrightarrow X$ the restriction of $p_i : Z_w \longrightarrow
X$ to $Y$ for $i = 1,2$. Then we have the commutative diagram
$$
\begin{CD}
Y @>{j}>> Z_w\\
@V{q_1}VV       @VV{p_1}V\\
\{x\} @>>{i_x}> X
\end{CD} \quad \ .
$$
Hence, $Y = \{x\} \times_X Z_w$. Then, by base change 
(\cite[Ch.~6, 8.4]{bdm}, \cite[Ch.~IV, Thm.~10.2]{dmod}), we have
$$
Li_x^+ \circ p_{1+} = q_{1+} \circ L j^+.
$$
This in turn implies that
\begin{align*}
(LT_x \circ LI_w)(\mathcal V^\cdot) & = (q_{1+} \circ Lj^+)(\mathcal T_w \otimes_{\mathcal O_{Z_w}} 
p_2^+(\mathcal V^\cdot))
= q_{1+}(j^*(\mathcal T_w)\otimes_{\mathcal O_Y} Lj^+(p_2^+(\mathcal V^\cdot))) \\
& = q_{1+}(j^*(\mathcal T_w)\otimes_{\mathcal O_Y} 
L(p_2 \circ j)^+(\mathcal V^\cdot))) = q_{1+}(j^*(\mathcal T_w)\otimes_{\mathcal O_Y} 
Lq_2^+(\mathcal V^\cdot)))
\end{align*} 
since $p_2^+$ is an exact functor.

\begin{lem}
We have
$$
j^*(\mathcal T_w) = \mathcal O_Y.
$$
\end{lem}
\begin{proof}
  Let $N_x$ be the unipotent subgroup of $\Int(\mathfrak g)$
  corresponding to $\mathfrak n_x$. Clearly, $Y$ is an $N_x$-orbit in
  $Z_w$ under the restriction of the action of $\Int(\mathfrak g)$ on
  $X \times X$.  Also, $j^*(\mathcal T_w)$ is an $N_x$-equivariant
  invertible $\mathcal O_Y$-module. Since $N_x$ is unipotent, it
  follows that it is isomorphic to $\mathcal O_Y$.
\end{proof}

Hence we finally get
$$
(LT_x \circ LI_w)(\mathcal V^\cdot) = q_{1+}(Lq_2^+(\mathcal V^\cdot))).
$$

Let $B_x$ be the Borel subgroup of $\Int(\mathfrak g)$ corresponding
to $x$.  Then $B_x$ acts on $X$ and its orbits are the Bruhat cells
$C(w)$, $w \in W$.  Let $i_w : C(w) \longrightarrow X$ be the natural
inclusion. Then the map $ k : y \longmapsto (x,y)$ is an isomorphism
of $C(w)$ onto $Y$. Moreover, we have $q_2 \circ k = p_2 \circ j \circ
k = i_w$. It follows that
$$
k^+(Lq_2^+(\mathcal V^\cdot))) = L(q_2 \circ k)^+(\mathcal V^\cdot) = 
Li_w^+(\mathcal V^\cdot)
$$
and
$$
Lq_2^+(\mathcal V^\cdot) = k_+(Li_w^+(\mathcal V^\cdot)).
$$
Hence
$$
(LT_x \circ LI_w)(\mathcal V^\cdot) = q_{1+} (k_+(Li_w^+(\mathcal V^\cdot))) =
(q_1 \circ k)^+(Li_w^+(\mathcal V^\cdot)).
$$
The map $q_1 \circ k$ is a map of $C(w)$ into the point $x$. If we
denote this map by $\pi_w : C(w) \longrightarrow \{pt\}$, we finally
get the following formula
$$
(LT_x \circ LI_w)(\mathcal V^\cdot) = \pi_{w,+}(Li_w^+(\mathcal V^\cdot)).
$$
Putting this together with the first part of the calculation, we
get the following result.

\begin{thm}
  \label{formula}
  Let $\lambda \in \theta$ be antidominant and regular. Then for any
  $x \in X$, $w \in W$, and $\mathcal U_\theta$-module $V$, we have
  $$
H_p(\mathfrak n_x,V)_{(w\lambda + \rho)} = H^{-p}(\pi_{w,+}(Li_w^+(D(\Delta_\lambda(V)))))
$$
for any $p \in \mathbb Z_+$.
\end{thm}

\section{Closed $K$-orbits in the equal rank case}
\label{closed_orbits}
Let $G_0$ be a connected semisimple Lie group with finite center. Let
$K_0$ be a maximal compact subgroup of $G_0$. In this section we
assume we are in the equal rank case; i.e., that $\rank G_0 = \rank
K_0$. Let $T_0$ be a maximal torus in $K_0$. Then $T_0$ is a compact
Cartan subgroup of $G_0$. Let $\mathfrak g$, $\mathfrak k$ and
$\mathfrak t$ be the complexified Lie algebras of $G_0$, $K_0$ and
$T_0$ respectively. Denote by $K$ the complexification of $K_0$. Then
the connected algebraic group $K$ acts naturally on the flag variety
$X$ of $\mathfrak g$. In this section we describe the structure of
closed $K$-orbits in $X$, and their stratification given by a Bruhat
stratification of $X$.

\subsection{Description of closed orbits}
Denote by $\sigma$ the Cartan involution of $\mathfrak g$ determined
by $\mathfrak k$. Then $\sigma$ fixes $\mathfrak k$. Therefore, it
also fixes $\mathfrak t$. Let $R$ be the root system in $\mathfrak
t^*$ of the pair $(\mathfrak g, \mathfrak t)$. Let $\alpha \in R$, and
$\xi$ in the root subspace $\mathfrak g_\alpha$. Then for any $\eta
\in \mathfrak t$, we have
$$
\alpha(\eta) \sigma(\xi) = \sigma(\alpha(\eta) \xi) = \sigma([\eta, \xi]) =
[\eta, \sigma(\xi)].
$$ Hence $\sigma(\xi)$ is also in $\mathfrak g_\alpha$. In particular,
it has to be proportional to $\xi$. Hence, $\sigma$ is either $1$ or
$-1$ on $\mathfrak g_\alpha$. A root $\alpha \in R$ is {\em compact
  (imaginary)} in the first case, and {\em noncompact (imaginary)} in
the second case. Denote by $R_c \subset R$ the set of all compact
roots. Since
$$
\mathfrak k = \mathfrak t \oplus \bigoplus_{\alpha \in R_c} \mathfrak g_\alpha,
$$
the set $R_c$ is naturally identified with the root system of
$(\mathfrak k,\mathfrak t)$. The complement $R - R_c = R_n$ is the set
of all noncompact roots.

It is well known that the group $K$ acts on $X$ with finitely many
orbits (see, for example, \cite[Ch. IV, Prop.~2.2]{book}). Since we
want to describe the $K$-orbits in the flag variety $X$ of $\mathfrak
g$, without any loss of generality we can assume that $K$ is a
subgroup of the group $\Int(\mathfrak g)$.

Let $\Omega$ be the subvariety of all $\sigma$-stable Borel
subalgebras in $X$.  Clearly, $\Omega$ is a union of $K$-orbits. More
precisely, by \cite[Lem.~6.16]{hmsw2}, $\Omega$ is the union of all
closed $K$-orbits in $X$. Hence, the closed $K$-orbits are the
connected components of $\Omega$.

Fix a closed $K$-orbit $Q$ in $X$ and a point $x \in Q$. Let $S_x$ be
the stabilizer of $x$ in $K$. Its Lie algebra $\mathfrak s_x$ is equal
to $\mathfrak k \cap \mathfrak b_x$. Hence, it is a solvable Lie
algebra. On the other hand, since $Q$ is a closed $K$-orbit, $Q$ is a
projective variety. Hence $S_x$ must be a parabolic subgroup of
$K$. It follows that $S_x$ is a Borel subgroup of $K$ and $\mathfrak
s_x = \mathfrak k \cap \mathfrak b_x$ is a Borel subalgebra of
$\mathfrak k$. Therefore the orbit map $K \ni k \longmapsto k \cdot x
\in Q$ factors through the flag variety $X_K$ of $\mathfrak k$. The
induced map $X_K \longrightarrow Q$ is an isomorphism which is the
inverse of the map $Q \ni y \longmapsto \mathfrak k \cap \mathfrak b_y
\in X_K$.

We have proved the following result.

\begin{thm}
  \label{closed}
  \begin{enumerate}
    \item[(i)] The variety $\Omega$ is the disjoint union of all
      closed $K$-orbits in the flag variety $X$.
    \item[(ii)] Let $Q$ be a closed $K$-orbit in $X$. Then the map
      $\mathfrak b_x \longmapsto \mathfrak b_x \cap \mathfrak k$
      defines a $K$-equivariant isomorphism of $Q$ with the flag
        variety $X_K$ of $\mathfrak k$.
     \end{enumerate}
\end{thm}

Now we remark that any closed $K$-orbit $Q$ contains a point $x$ such
that $\mathfrak b_x$ contains the Cartan subalgebra $\mathfrak t$. Let
$y$ be a point in $Q$. Then $\mathfrak b_y \cap \mathfrak k$ is a Borel
subalgebra of $\mathfrak k$. Let $\mathfrak c$ be a Cartan subalgebra
of $\mathfrak k$ contained in $\mathfrak b_y$. Then, since all Cartan
subalgebras of $\mathfrak k$ are $K$-conjugate, $\mathfrak c$ is
$K$-conjugate of $\mathfrak t$. It follows that a $K$-conjugate
$\mathfrak b_x$ of $\mathfrak b_y$ contains $\mathfrak t$ and $x
\in Q$.

Let $\mathfrak b$ be a Borel subalgebra of $\mathfrak k$ containing
$\mathfrak t$. Then any Borel subalgebra of $\mathfrak g$ containing
$\mathfrak b$ contains the Cartan subalgebra $\mathfrak t$ and it is
$\sigma$-stable. Therefore, it is in $\Omega$. Moreover, by Theorem
\ref{closed}, the number of closed $K$-orbits is equal to the
cardinality of the set $\{ x \in X \mid \mathfrak b_x \cap \mathfrak k
= \mathfrak b \}$. This number is equal to the number of sets of
positive roots $R^+$ in $R$ which contain a fixed set of positive
roots $R_c^+$ in $R_c$. It follows that it is equal to the number of
Weyl chambers of $R$ contained in a fixed Weyl chamber of $R_c$.  Let
$W_K$ be the Weyl group of $\mathfrak k$. By specialization at $x$, it
can be identified with a subgroup of $W$. This immediately implies
that the number of closed $K$-orbits is equal to $\Card (W/W_K) =
\Card(W)/\Card(W_K)$.

\begin{pro}
  \label{count_closed}
The number of closed $K$-orbits in $X$ is equal to $\Card(W/W_K)$.
\end{pro}

\subsection{A stratification of closed $K$-orbits}
Let $\mathfrak b$ be a Borel subalgebra containing $\mathfrak t$.
Then it is $\sigma$-stable and therefore determines a point in a
closed $K$-orbit in $X$. Moreover, $\mathfrak b \cap \mathfrak k$ is a
Borel subalgebra of $\mathfrak k$ containing $\mathfrak t$. We
consider the specialization $\mathfrak h^* \longrightarrow \mathfrak
t^*$ determined by $\mathfrak b$. This defines an isomorphism of $W$
with the Weyl group of the root system $R$ in $\mathfrak
t^*$. Moreover, it identifies $W_K$ with a subgroup of $W$ generated
by reflections with respect to compact roots. The set of positive roots
$\Sigma^+$ determines a set of compact positive roots. This set
determines a set of simple roots in the root system of compact roots.
Corresponding reflections generate $W_K$ and define the length
function $\ell_K$ on $W_K$.

Let $B$ be the Borel subgroup of $\Int(\mathfrak g)$ corresponding to
$\mathfrak b$. Denote by $C(w)$ the $B$-orbit in $X$ corresponding the
element $w \in W$. Let $B_K$ be the Borel subgroup of
$K$ corresponding to $\mathfrak b \cap \mathfrak k$. Then $B_K$
is a subgroup of $B$.

We denote by $C_K(w)$, $w \in W_K$, the Bruhat cells in the flag
variety $X_K$ of $\mathfrak k$ corresponding to the action of
$B_K$. As we proved in Theorem \ref{closed}.(ii), for any closed $K$-orbit
$Q$, the map $x \longmapsto \mathfrak b_x \cap \mathfrak k$ is a
$K$-equivariant isomorphism of $Q$ onto $X_K$. Clearly, for any $w \in
W_K$, it induces by restriction an isomorphism of a $B_K$-orbit
$D_Q(w)$ in $Q$ onto $C_K(w)$. Therefore, $Q$ is the union of
$\Card(W_K)$ $B_K$-orbits. By Theorem \ref{closed}.(i) and Proposition
\ref{count_closed}, we see that $\Omega$ is the union of $\Card(W)$
$B_K$-orbits. Clearly, the Borel subalgebras of $\mathfrak g$
containing $\mathfrak t$ represent all Bruhat cells $C(w)$, $w \in W$,
in $X$ with respect to the action of $B$. Therefore, the intersection
$\Omega \cap C(w)$ is nonempty for any $w \in W$.  Moreover, for any
$w \in W$, the intersection $\Omega \cap C(w)$ is $B_K$-invariant and
therefore a union of $B_K$-orbits. This in turn implies that $\Omega
\cap C(w)$ is a $B_K$-orbit for any $w \in W$.

Let $Q$ be a closed $K$-orbit. By the above discussion and Theorem
\ref{closed}, we see that there exists a unique point $x \in Q$ which
is fixed by $B_K$. Therefore, there exists a unique $u \in W$ such that
$Q \cap C(u) = \{x\}$. We summarize this discussion in the following
lemma.

\begin{lem}
  \label{clorb_strat}
  Let $Q$ be a closed $K$-orbit in $X$.
 \begin{enumerate}
\item[(i)]
  There is a unique $u \in
  W$ such that $Q \cap C(u)$ is a point. 
\item[(ii)]
  The orbit $Q$ intersects Bruhat cell $C(v)$, $v \in W$, if
  and only if $v = w u$ for some $w \in W_K$.
\item[(iii)]
  The variety $Q \cap C(wu)$ is a $B_K$-orbit in $Q$ for
    any $w \in W_K$. More precisely, we have $D_Q(w) = Q \cap  C(wu)$
    for all $w \in W_K$.
  \item[(iv)] We have $\dim D_Q(w) = \ell_K(w)$.
  \end{enumerate} 
\end{lem}

\section{Calculation of $\mathfrak n$-homology for the discrete series}
\label{nhomology}
Using Theorem \ref{formula}, we can prove a result of Wilfried Schmid on
$\mathfrak n$-homology of discrete series representations
\cite{l2coho}. To illustrate the simplicity of our argument, we first
treat the special case of irreducible finite-dimensional
representations of a connected compact semisimple Lie groups due to
Bertram Kostant \cite{kostant}. The method of the proof in both cases
is essentially identical.

\subsection{Kostant's theorem}
Let $\mathfrak b$ be a Borel subalgebra of $\mathfrak g$. Let
$\mathfrak n$ be the nilpotent radical of $\mathfrak b$. Let $F$ be an
irreducible finite-dimensional representation with lowest weight
$\lambda$. Then, by the Borel-Weil theorem, we have $F = \Gamma(X, \mathcal
O(\lambda))$ \cite[Ch.~II, Thm. 5.1]{book}. Moreover, $\mathcal
O(\lambda)$ is a $\mathcal D_{\lambda - \rho}$-module. As we remarked
before, since $\lambda - \rho$ is regular, we have
$$
H_p(\mathfrak n,F) 
= \bigoplus_{w \in W} H_p(\mathfrak n, F)_{(w(\lambda - \rho)+\rho)}.
$$ 
Moreover, we have $\Delta_{\lambda - \rho}(F) = \mathcal
O(\lambda)$.  Therefore, it follows that $Li_w^+(D(\mathcal O(\lambda)))
= D(\mathcal O_{C(w)})$. 
This implies that 
$$ 
H_p(\mathfrak n, F)_{(w(\lambda - \rho)+\rho)} =
H^{-p}(\pi_{w,+}(D(\mathcal O_{C(w)}))).
$$

In addition, we recall the following simple result.
\begin{lem}
\label{deRham}
Let $Y$ be an affine space. Let $\pi : Y \longrightarrow \{ pt\}$ be
the map of $Y$ into the point. Then we have
$$
\pi_+(D(\mathcal O_Y)) = D(\mathbb C)[\dim Y].
$$
\end{lem}

\begin{proof}
This follows from \cite[Ch.~I, Thm.~11.2 and Lem. 11.4]{dmod}.
\end{proof}

This implies that $H_p(\mathfrak n, F)_{(w(\lambda - \rho)+\rho)} =
0$ if $p \ne \ell(w)$ and $H_{\ell(w)}(\mathfrak n,F)_{(w(\lambda -
  \rho) + \rho)} = \mathbb C$.

Putting this all together, we get
$$
H_p(\mathfrak n,F) = \bigoplus_{w \in W(p)} \mathbb C_{w(\lambda - \rho)+\rho}
$$
where $W(p)$ is the subset of $W$ consisting of all elements $w$ such that
$\ell(w) = p$. We have proven the following result.

\begin{thm}[Kostant]
  \label{kostant}
Let $F$ be an irreducible finite-dimensional representation with
lowest weight $\lambda$. Then
$$
H_p(\mathfrak n,F) = \bigoplus_{w \in W(p)} \mathbb C_{w(\lambda - \rho)+\rho}
$$
for $p \in \mathbb Z_+$.
\end{thm}
This is Kostant's result mentioned above.

\subsection{Schmid's result}
\label{schmidhomo}
Now we discuss a generalization of Kostant's result corresponding to
the $\mathfrak n$-homology of discrete series representations. This
result has been proved by Schmid \cite[Thm.~4.1]{l2coho}.

We first recall the geometric version of the classification of
discrete series representations of a connected semisimple Lie group
$G_0$ with finite center \cite{hmsw2}. First, the discrete series
representations exist if and only if the rank of $G_0$ is equal to the
rank $K_0$. In this situation, we follow the notation from the
introduction to the preceding section.

Let $V$ be the Harish-Chandra module of a discrete series representation
of $G_0$. Assume that $V$ is in $\mathcal M(\mathcal U_\theta)$. Then
$\theta$ is a regular $W$-orbit of a unique real strongly antidominant
$\lambda \in \mathfrak h^*$; i.e., $\alpha\check{\ }(\lambda) < 0$ for
any $\alpha$ in $\Sigma^+$ (see \cite[Thm.~12.4]{hmsw2}). Moreover,
there exists a unique closed $K$-orbit $Q$ and a unique irreducible
$K$-equivariant connection on $Q$ compatible with $\lambda+\rho$ such
that $V = \Gamma(X,\mathcal I(Q,\tau))$ \cite[Thm.~12.5]{hmsw2}.

Let $\mathfrak b$ be a $\sigma$-stable Borel subalgebra of $\mathfrak
g$. Let $\mathfrak n = [\mathfrak b,\mathfrak b]$.

Denote by $i_Q : Q \longrightarrow X$ the natural immersion of $Q$,
and, for any $v \in W$, by $i_v : C(v) \longrightarrow X$ the natural
immersion of the Bruhat cell (with respect to $\mathfrak b$) $C(v)$
into $X$.  Since $Q$ is the support of $\mathcal I(Q,\tau)$, the
restriction of $\mathcal I(Q,\tau)$ to the complement of $Q$ is
$0$. Therefore $Li_v^+(\mathcal I(Q,\tau)) = 0$ for any $v$ such that
$C(v) \cap Q = \emptyset$. Hence, in this case, $H_p(\mathfrak
n,\Gamma(X,\mathcal I(Q,\tau)))_{(v\lambda + \rho)} = 0$ for $p \in
\mathbb Z_+$, by Theorem \ref{formula}.

Let $v \in W$ be such that $Q \cap C(v) \ne \emptyset$.  As we
explained in Lemma \ref{clorb_strat}, there exists a unique $u \in W$
such that $Q \cap C(u)$ is a point.  Then $v = w u$, for some $w \in
W_K$, and $D_Q(w)$ is a smooth subvariety of $Q$. Denote by $a :
D_Q(w) \longrightarrow C(wu)$ and $b : D_Q(w) \longrightarrow Q$ the
natural immersions. Then we have the commutative diagram
$$
\begin{CD}
D_Q(w) @>{a}>> C(wu)\\
@V{b}VV       @VV{i_{wu}}V\\
Q @>>{i_Q}> X
\end{CD}
$$
i.e., $D_Q(w)$ is the fiber product $Q \otimes_X C(wu)$. By base change 
\cite[Ch.~IV, Thm.~10.2]{dmod}, we have
$$
Ri_{wu}^! \circ i_{Q,+} = a_+ \circ R b^!.
$$ 
Moreover, we have $Li_{wu}^+ = Ri_{wu}^![\dim X - \dim C(wu)]$ and
$Lb^+ = Rb^![\dim Q - \dim D_Q(w)]$.  Therefore, we have
\begin{align*}
Li_{wu}^+(D(\mathcal I(Q,\tau)))& = Li_{wu}^+(D(i_{Q,+}(\tau))) =
 Li_{wu}^+(i_{Q,+}(D(\tau))) \\
&= Ri_{wu}^!(i_{Q,+}(D(\tau)))[\dim X - \ell(wu)]\\
&= a_+(R b^!(D(\tau)))[\dim X - \ell(wu)] \\
&= a_+(Lb^+(D(\tau)))[\dim X - \ell(wu)][-\dim Q + \dim D_Q(w)]\\
&= a_+(Lb^+(D(\tau)))[\dim X - \dim X_K - \ell(wu) + \ell_K(w)]
\end{align*}
by Lemma \ref{clorb_strat}.(iv).
Since $\tau$ is a connection on $Q$, we have
$$
Li_{wu}^+(D(\mathcal I(Q,\tau))) 
= a_+(D(b^+(\tau)))[\dim X - \dim X_K - \ell(wu) + \ell_K(w)].
$$
Denote by $c$ the map of $D_Q(w)$ into a point $\{pt\}$.
Using Lemma \ref{deRham} this leads to
\begin{align*}
\pi_{w,+}(Li_{wu}^+(D(\mathcal I(Q,\tau& )))) 
= \pi_{w,+}(a_+(D(b^+(\tau)))[\dim X - \dim X_K - \ell(wu) + \ell_K(w)] \\ 
&= (\pi_w \circ a)_+(D(\mathcal O_{D_Q(w)}))[\dim X - \dim X_K - \ell(wu) 
+ \ell_K(w)] \\
&= c_+(D(\mathcal O_{D_Q(w)}))[\dim X - \dim X_K - \ell(wu) + \ell_K(w)] \\
&=
D(\mathbb C)[\dim X - \dim X_K - \ell(wu) + \ell_K(w)][\ell_K(w)] \\
&= D(\mathbb C)[\dim X - \dim X_K - \ell(wu) + 2 \ell_K(w)].
\end{align*}
Finally, $\dim X$ is equal to the number of all positive roots; i.e.,
to half of the number of all roots. Analogously, $\dim X_K$ is equal
to the half of the number of all compact roots. Hence, the difference
$\dim X - \dim X_K$ is equal to the half of the number of all
noncompact roots, i.e., $ \frac 1 2 \dim(\mathfrak g/\mathfrak k)$.

Applying again Theorem \ref{formula}, we have
\begin{align*}
H_p(\mathfrak n,\Gamma(X,\mathcal I(Q,\tau)))_{(wu\lambda + \rho)} & =
H^{-p}(\pi_{w,+}(Li_{wu}^+(D(\mathcal I(Q,\tau))))) \\
& = \begin{cases}
  0  & \text{ if } p \ne  \frac 1 2 \dim(\mathfrak g/\mathfrak k) - \ell(wu) +
  2 \ell_K(w) ; \\
  \mathbb C & \text{ if } p =  \frac 1 2 \dim(\mathfrak g/\mathfrak k) -
  \ell(wu) + 2 \ell_K(w)
\end{cases}
\end{align*}
for any $w \in W_K$.

This proves the following result.\footnote{Clearly, if $G_0$ is
compact, this specializes to Kostant's result.}

\begin{thm}[Schmid]
  \label{schmid}
  Let $V$ be a discrete series representation such that $V =
  \Gamma(X,\mathcal I(Q,\tau))$.

Then, if $v \notin W_Ku$,
$$
H_p(\mathfrak n,V)_{(v \lambda + \rho)} = 0,
$$ 
for all $p \in \mathbb Z_+$. Moreover, for $w \in W_K$, we have
$$
H_p(\mathfrak n,V)_{(wu\lambda + \rho)}
= \begin{cases}
  0  & \text{ if } p \ne  \frac 1 2 \dim(\mathfrak g/\mathfrak k) - \ell(wu) +
  2 \ell_K(w) ; \\
  \mathbb C & \text{ if } p =  \frac 1 2 \dim(\mathfrak g/\mathfrak k) -
  \ell(wu) + 2 \ell_K(w).
\end{cases}
$$
\end{thm}

\subsection{Kostant's result and the BGG resolution}
\label{bgg}
In this section we interpret the calculation of $\mathfrak n$-homology
of finite-dimensional representations using the BGG resolution. This gives
us a more precise version of the result, we see that each cohomology
class corresponds to a Bruhat cell in $X$.

Let $\lambda$ be an antidominant weight. Then, the $\mathcal
D_{\lambda - \rho}$-module $\mathcal O(\lambda)$ can be represented in
$D^b(\mathcal D_{\lambda-\rho})$ by its Cousin resolution
corresponding to the stratification of $X$ by Bruhat cells $C(w)$ as
explained in Appendix \ref{cousinres}.

Let $C(w)$ be a Bruhat cell in $X$ and $i_w : C(w) \longrightarrow X$
the natural inclusion. The cell $C(w)$ admits unique irreducible
$N$-equivariant connection $\mathcal O_{C(w)}$. Its direct image
$\mathcal I(w,\lambda - \rho) = i_{w,+}(\mathcal O_{C(w)})$ is a
{\em standard} $\mathcal D_{\lambda-\rho}$-module attached to $C(w)$
\cite[Ch.~V, Sec.~1]{book}.

The Cousin resolution (see Theorem \ref{twcousin} in Appendix) is a complex
$\mathcal C^\cdot$ such that
$$
\mathcal C^p = \bigoplus_{w \in W(\dim X - p)} \mathcal I(w,\lambda - \rho)
$$
for any $p \in \mathbb Z_+$, with explicitly given differentials.

Clearly, there is a natural monomorphism from $\mathcal O(\lambda)$ into
the standard $\mathcal D_{\lambda-\rho}$-module $\mathcal
I(w_0,\lambda - \rho)$ attached to the open Bruhat cell $C(w_0)$
(where $w_0$ is the longest element of the Weyl group $W$).
Therefore, there is a natural morphism $\epsilon$ of $D(\mathcal
O(\lambda))$ into $\mathcal C^\cdot$ in $D^b(\mathcal D_{\lambda -
  \rho})$.  Theorem \ref{twcousin} says that $\epsilon$ is
an isomorphism.

Let $\theta$ be the Weyl group orbit of $\lambda - \rho$. Since the
functor $\Gamma$ is exact for antidominant $\lambda$, $C^\cdot =
\Gamma(X,\mathcal C^\cdot)$ is isomorphic to $D(F)$ in $D^b(\mathcal
U_\theta)$ where $F$ is the finite-dimensional $\mathcal U(\mathfrak
g)$-module with lowest weight $\lambda$; i.e., we get a resolution of
$F$ by modules $C^p$, $p \in \mathbb Z_+$. By \cite[Ch.~V,
  1.14]{book}, we know that
$$
C^p = \bigoplus_{w \in W(\dim X - p)} I(w(\lambda - \rho))
$$ 
for all $p \in \mathbb Z_+$. Here $I(\mu)$ are the {\em duals} of the
Verma modules $M(\mu)$ in the category of highest weight modules.
This is the {\em BGG resolution} of $F$. \footnote{More precisely, this is the
dual of the original BGG resolution.}

By Theorem \ref{formula}, we have
$$
H_p(\mathfrak n,I(w(\lambda - \rho)))_{(v(\lambda-\rho)+\rho)} 
= H^{-p}(\pi_{v,+}(Li_v^+(D(\mathcal I(w,\lambda-\rho))))).
$$
By base change,
$$
Li_v^+(D(\mathcal I(w,\lambda-\rho))) = 0 
$$
if $v \ne w$. Hence, we have
$$
H_p(\mathfrak n,I(w(\lambda - \rho)))_{(v(\lambda-\rho)+\rho)}  = 0
$$ 
for all $p \in \mathbb Z_+$ if $v \ne w$.

On the other hand, if $v = w$ we have $ Li_w^+
= Ri_w^![\dim X - \ell(w)]$, and
$$ 
Li_w^+(D(\mathcal I(w,\lambda-\rho))) 
= Ri_w^!(D(\mathcal I(w,\lambda-\rho)))[\dim X - \ell(w)] = 
D(\mathcal O_{C(w)})[\dim X - \ell(w)]
$$
by Kashiwara's equivalence of categories.  By Lemma \ref{deRham},
this implies that
$$
\pi_{w,+}(Li_w^+(D(\mathcal I(w,\lambda-\rho)))) = D(\mathbb C)[\dim X].
$$
Hence, we conclude that
$$
H_p(\mathfrak n,I(w(\lambda - \rho)))_{(w(\lambda-\rho)+\rho)} = \begin{cases}
\mathbb C & \text{ if } p = \dim X; \\
0         & \text{ if } p \ne \dim X.
\end{cases}
$$
Therefore, we proved the following result.

\begin{lem}
Let $\lambda$ be an antidominant weight. Then 
$$
H_p(\mathfrak n,I(w(\lambda - \rho))) = \begin{cases}
\mathbb C_{w(\lambda-\rho)+\rho} & \text{ if } p = \dim X; \\
0         & \text{ if } p \ne \dim X.
\end{cases}
$$
\end{lem}

This implies the following result.

\begin{cor}
\label{ndv}
Let $\lambda$ be an antidominant weight. Then 
$$
H_p(\mathfrak n, C^q) = \begin{cases}
\bigoplus_{w \in W(\dim X - q)}\mathbb C_{w(\lambda-\rho)+\rho} & \text{ if } p = \dim X; \\
0         & \text{ if } p \ne \dim X.
\end{cases}
$$
\end{cor}

By the above discussion, the $\mathfrak n$-homology of $F$ is given
by the hypercohomology of the $\mathfrak n$-homology functor for the
complex $C^\cdot$. 

More precisely, the $\mathfrak n$-homology is the left derived functor
of the functor $V \longmapsto V/\mathfrak n V$ from $\mathcal
M(\mathcal U_\theta)$ into the category $\mathcal M_{ss}(\mathcal
U(\mathfrak h))$ of semisimple $\mathcal U(\mathfrak h)$-modules. We
can view this functor as an exact functor from $D^b(\mathcal
U_\theta)$ into $D^b(\mathcal M_{ss}(\mathcal U(\mathfrak h)))$.

Let $w \in W$. Then we can consider the composition of this functor
with the functor of taking the weight subspace of the weight
$w(\lambda - \rho) +\rho$. This is an exact functor $S$ from
$D^b(\mathcal U_\theta)$ into the bounded derived category
of vector spaces $D^b(\mathbb C)$. Clearly, we have
$$
H_p(\mathfrak n,V)_{(w(\lambda - \rho) + \rho)} = H_{-p}(S(D(V))).
$$
Therefore, we see that
\begin{align*} 
H_{-p}(S(D(C^q))) 
& = H_p(\mathfrak n,C^q)_{(w(\lambda - \rho) + \rho)} \\ 
& = \begin{cases}
0 & \text{ if $p \ne \dim X$ or $q \ne \dim X - \ell(w)$}; \\
\mathbb C & \text{ if $p = \dim X$ and $q = \dim X - \ell(w)$}.
\end{cases}
\end{align*}
It follows that 
$$
S(D(C^q)) = 0
$$
if $q \ne \dim X - \ell(w)$; and
$$
S(D(C^{\dim X - \ell(w)})) = D(\mathbb C)[\dim X].
$$

We interrupt the proof to prove a simple result in homological
algebra. An interested reader can easily supply an alternate argument
via spectral sequences.

Let $\mathcal A$ and $\mathcal B$ be two abelian categories. We denote
by $D^b(\mathcal A)$ and $D^b(\mathcal B)$ their bounded derived
categories.  Let $F : D^b(\mathcal A) \longrightarrow D^b(\mathcal B)$
be an exact functor between these triangulated categories.

\begin{lem}
\label{homol}
Let $C^\cdot$ be a complex in $D^b(\mathcal A)$. Assume that there
exists an integer $p_0$ such that $F(D(C^p)) = 0$ for all $p \ne p_0$.
Then 
$$
F(C^\cdot) = F(D(C^{p_0}))[-p_0].
$$
\end{lem}

\begin{proof}
We use freely the results on stupid truncations from \cite[Ch.~III,
  4.4 and 4.5]{dercat}.  

For any complex $X^\cdot$ in $D^b(\mathcal A)$ and an integer $s$ we
denote by $\sigma_{\ge s}(X^\cdot)$ and $\sigma_{\le s}(X^\cdot)$ its
stupid truncations. Then we have the distinguished triangle of stupid
truncations
$$
\xymatrix@R=60pt{
& \sigma_{\le s-1}(X^\cdot) \ar[dl]^{[1]} \\
\sigma_{\ge s}(X^\cdot) \ar[rr] & & X^\cdot 
\ar[ul]
} \quad .
$$

The proof is by induction on the length $\ell(C^\cdot)$. First we
remark that if the statement holds for complex $C^\cdot$, then for
such complex $F(D(C^p)) = 0$ for all $p \in \mathbb Z$ implies
$F(C^\cdot) = 0$.

The proof is by induction on the length of $C^\cdot$. If $C^p = 0$,
for all $p \in \mathbb Z$, the claim is obvious. If $\ell(C^\cdot) =
1$, we have $C^\cdot = D(C^q)[-q]$ for some $q \in \mathbb
Z$. Therefore, we have
$$
F(C^\cdot) = F(D(C^q)[-q]) = F(D(C^q))[-q]
$$
and the statement holds.

If $\ell(C^\cdot) > 1$, there exists $s \in \mathbb Z$ such that the
lengths of $\sigma_{\le s-1}(C^\cdot)$ and $\sigma_{\ge s}(C^\cdot)$
are strictly less than $\ell(C^\cdot)$. Moreover, either $p_0 < s$ or
$p_0 \ge s$. In the first case, by the induction assumption, $F(\sigma_{\le
  s-1}(C^\cdot)) = F(D(C^{p_0}))[-p_0]$. Moreover, by the above claim
$F(\sigma_{\ge s}(C^\cdot)) = 0$. From the distinguished triangle of
stupid truncations we conclude that $F(C^\cdot) = F(\sigma_{\le
  s-1}(C^\cdot)) = F(D(C^{p_0}))[-p_0]$.

In the second case, by the induction assumption, $F(\sigma_{\ge
  s}(C^\cdot)) = F(D(C^{p_0}))[-p_0]$. Moreover, by the above claim
$F(\sigma_{\le s-1}(C^\cdot)) = 0$. From the distinguished triangle of
stupid truncations we conclude that $F(C^\cdot) = F(\sigma_{\ge
  s}(C^\cdot)) = F(D(C^{p_0}))[-p_0]$.
\end{proof}

Now we go back to the discussion of $\mathfrak n$-homology. By Lemma
\ref{homol} it follows that
$$ 
S(C^\cdot) = D(\mathbb C)[\dim X][-\dim X +
  \ell(w)] = D(\mathbb C)[\ell(w)].
$$
Since $D(F)$ is isomorphic to its BGG resolution $C^\cdot$ in 
$D^b(\mathcal U_\theta)$, we see that $S(D(F)) =
D(\mathbb C)[\ell(w)]$. This implies that
\begin{align*}
H_p(\mathfrak n,F)_{(w(\lambda - \rho) + \rho)} &= H_{-p}(S(D(F))) 
= H_{-p}(S(C^\cdot))\\ 
&= H_{-p}(D(\mathbb C)[\ell(w)]) 
 = \begin{cases}
0 & \text{ if $p \ne \ell(w)$}; \\
\mathbb C & \text{ if $p = \ell(w)$}.
\end{cases}
\end{align*}
This finally leads to
$$
H_p(\mathfrak n,F)
= \bigoplus_{w \in W} H_p(\mathfrak n,F)_{(w(\lambda - \rho) + \rho)}
= \bigoplus_{w \in W(p)} \mathbb C_{w(\lambda - \rho) + \rho} ,
$$
i.e., Kostant's theorem (Theorem \ref{kostant}).

Therefore, this calculation shows that each dual Verma module in the
BGG resolution of $F$ gives exactly one cohomology class in $\mathfrak
n$-homology of $F$.

\subsection{Schmid's result and the Trauber resolution}
In this section we interpret the calculation of $\mathfrak n$-homology
of discrete series representations using the Trauber resolution
\cite{zuck}. This argument is very similar to that in Section
\ref{bgg}.

Let $\lambda$ be antidominant and regular. Fix a closed $K$-orbit
$Q$. Then, the $\mathcal D_\lambda$-module $\mathcal I(Q,\tau)$ can be
represented in $D^b(\mathcal D_\lambda)$ by its Cousin resolution
corresponding to the stratification of $Q$ by $B_K$-orbits $D_Q(w)$ as
explained in the Appendix \ref{itwcousin}.

Let $D_Q(w)$, $w \in W_K$, be a $B_K$-orbit in $Q$. Denote by $N_K$
the unipotent radical of $B_K$. Then $D_Q(w)$ admits unique
irreducible $N_K$-equivariant connection $\mathcal O_{D_Q(w)}$. Its
direct image $\mathcal J(w,\lambda)$ is a standard $\mathcal
D_\lambda$-module attached to $D_Q(w)$.

The Cousin resolution of $\mathcal I(Q,\tau)$ from Theorem
\ref{itwcousin} is a complex $\mathcal D^\cdot$ such that
$$
\mathcal D^p = \bigoplus_{w \in W_K(\dim Q - p)} \mathcal J(w,\lambda)
$$
for any $p \in \mathbb Z_+$, with explicitly given differentials.

Clearly, there is a natural monomorphism $\mathcal I(Q,\tau)$ into the
standard $\mathcal D_\lambda$-module $\mathcal J(w_0,\lambda)$
attached to the $B_K$-orbit $D_Q(w_0)$ open in $Q$ (here $w_0$ is the
longest element of the Weyl group $W_K$ now).  Therefore, there is a
natural morphism $\epsilon$ of $D(\mathcal I(Q,\tau))$ into $\mathcal
D^\cdot$ in $D^b(\mathcal D_\lambda)$.  The main result of
\ref{itwcousin} says that $\epsilon$ is an isomorphism.

Let $\theta$ be the Weyl group orbit of $\lambda$. Since the functor
$\Gamma$ is exact for antidominant $\lambda$, $D^\cdot =
\Gamma(X,\mathcal D^\cdot)$ is isomorphic to $D(\Gamma(X,\mathcal
I(Q,\tau)))$ in $D^b(\mathcal U_\theta)$, i.e., we get a resolution of
$\Gamma(X,\mathcal I(Q,\tau))$ by modules $D^p$, $p \in \mathbb Z_+$.
This is the {\em Trauber resolution} of the discrete series
$\Gamma(X,\mathcal I(Q,\tau))$ \cite{zuck}.

We put
$$
J(w,\lambda) = \Gamma(X,\mathcal J(w,\lambda))
$$
for any $w \in W_K$.

By Theorem \ref{formula}, we have
$$
H_p(\mathfrak n,J(w,\lambda))_{(v\lambda + \rho)} 
= H^{-p}(\pi_{v,+}(Li_v^+(D(\mathcal J(w,\lambda)))))
$$ for any $v \in W$.  By base change, we get
$$
Li_v^+(D(\mathcal J(w,\lambda))) = 0 
$$
if $v \ne wu$.

Hence, if $v \ne wu$, we have
$$
H_p(\mathfrak n,J(w,\lambda))_{(v\lambda+\rho)}  = 0
$$ 
for all $p \in \mathbb Z_+$.

On the other hand, if $v = wu$ we have 
$$ 
Li_{wu}^+(D(\mathcal J(w,\lambda))) 
= Ri_{wu}^!(D(\mathcal J(w,\lambda)))[\dim X - \ell(wu)].  
$$
Denote by $j_w$ the immersion of $D_Q(w)$ into $Q$. Then we have
$$
D(\mathcal J(w,\lambda)) = i_{Q,+}(j_{w,+}(D(\mathcal O_{D_Q(w)}))),
$$
since $D_Q(w)$ is affine and $Q$ is affinely imbedded. Therefore,
using the commutative diagram from Section \ref{schmidhomo}, we have
\begin{align*}
Ri_{wu}^!(D(\mathcal J(w,\lambda))) 
& = Ri_{wu}^!(i_{Q,+}(j_{w,+}(D(\mathcal O_{D_Q(w)}))))\\
& = a_+(Rb^!(j_{w,+}(D(\mathcal O_{D_Q(w)})))) = a_+(D(\mathcal O_{D_Q(w)}))
\end{align*}
by base change and Kashiwara's equivalence of categories. If we denote $c =
\pi_{wu} \circ a$ by Lemma \ref{deRham}, this implies that
\begin{align*}
\pi_{wu,+}(Li_{wu}^+(D(\mathcal J(w,\lambda)))) 
& = \pi_{wu,+}(a_+(D(\mathcal O_{D_Q(w)})))[\dim X - \ell(wu)]\\
& = c_+(D(\mathcal O_{D_Q(w)}))[\dim X - \ell(wu)] \\
& = D(\mathbb C)[\dim X - \ell(wu)][\ell_K(w)]\\
&= D(\mathbb C)[\dim X - \ell(wu) + \ell_K(w)].
\end{align*}
Hence, we conclude that
$$
H_p(\mathfrak n,J(w,\lambda))_{(wu\lambda+\rho)} = \begin{cases}
\mathbb C & \text{ if } p = \dim X - \ell(wu) + \ell_K(w); \\
0         & \text{ if } p \ne \dim X - \ell(wu) + \ell_K(w).
\end{cases}
$$
Therefore, we proved the following result.

\begin{lem}
  Let $\lambda$ be an antidominant weight and $w \in W_K$. Then
$$
H_p(\mathfrak n,J(w,\lambda)) = \begin{cases}
\mathbb C_{wu\lambda+\rho} & \text{ if } p = \dim X - \ell(wu) + \ell_K(w) ; \\
0         & \text{ if } p \ne \dim X- \ell(wu) + \ell_K(w) .
\end{cases}
$$
\end{lem}

By the above discussion, the $\mathfrak n$-homology of
$\Gamma(X,\mathcal I(Q,\tau))$ is given by the hypercohomology of the
$\mathfrak n$-homology functor for the complex $D^\cdot$.

Let $v \in W$. As in Section \ref{bgg}, we consider the exact functor
$S$ from $D^b(\mathcal U_\theta)$ into bounded derived category of
vector spaces $D^b(\mathbb C)$.

Therefore, we see that
\begin{multline*} 
H_{-p}(S(D(D^q))) 
= H_p(\mathfrak n,D^q)_{(v\lambda + \rho)} \\ 
= \begin{cases}
0 & \text{ if $v \ne wu$, $w \in W_K$, or $p \ne \dim X - \ell(wu) + \ell_K(w)$
 or $q \ne \dim Q - \ell_K(w)$}; \\
\mathbb C & \text{ if $v = wu$, $w \in W_K$, $p = \dim X- \ell(wu) + \ell_K(w)$
 and $q = \dim Q - \ell_K(w)$}.
\end{cases}
\end{multline*}

If $v \notin W_K u$, we see that $S(D(D^q)) = 0$. By Lemma \ref{homol} it
follows that $S(D^\cdot) = 0$. Since $D(\Gamma(X,\mathcal I(Q,\tau)))$
is isomorphic to its Trauber resolution $D^\cdot$ in $D^b(\mathcal
U_\theta)$, we see that $S(D(\Gamma(X,\mathcal I(Q,\tau)))) = 0$.
Therefore, we see that $H_p(\mathfrak n, \Gamma(X,\mathcal
I(Q,\tau)))_{(v\lambda + \rho)} = 0$ if $v \notin W_K u$ for all $p \in
\mathbb Z_+$.

Assume that $v = wu$, $w \in W_K$. It follows that 
$$
S(D(D^q)) = 0
$$
if $q \ne \dim Q - \ell_K(w)$; and
$$
S(D(D^{\dim Q - \ell_K(w)})) = D(\mathbb C)[\dim X - \ell(wu) + \ell_K(w)].
$$
By Lemma \ref{homol} it follows that
\begin{align*} 
S(D^\cdot) &= D(\mathbb C)[\dim X - \ell(wu) + \ell_K(w)][-\dim Q +
\ell_K(w)]\\
&= D(\mathbb C)[\dim X - \dim X_K -  \ell(wu) + 2 \ell_K(w)] \\
&=
D(\mathbb C)[ \tfrac 1 2 \dim (\mathfrak g/\mathfrak k)
  - \ell(wu) + 2 \ell_K(w)].
\end{align*} 
Since $D(\Gamma(X,\mathcal I(Q,\tau)))$ is isomorphic to its
Trauber resolution $D^\cdot$ in $D^b(\mathcal U_\theta)$, we see that
$S(D(\Gamma(X,\mathcal I(Q,\tau)))) = D(\mathbb C)
[\frac 1 2 \dim (\mathfrak g/\mathfrak k) - 
\ell(wv) + 2 \ell_K(w)]$. This implies that
\begin{align*}
H_p(\mathfrak n,\Gamma(X,\mathcal I(Q,\tau)))_{(wu\lambda + \rho)} 
&= H_{-p}(S(D(\Gamma(X,\mathcal I(Q,\tau))))) 
= H_{-p}(S(D^\cdot))\\ 
&= H_{-p}(D(\mathbb C) [ \tfrac 1 2 \dim (\mathfrak g/\mathfrak k)
  -  \ell(wu) + 2 \ell_K(w) ])\\ 
&= \begin{cases}
  0 & \text{ if $p \ne \frac 1 2 \dim(\mathfrak g/\mathfrak k) -  \ell(wu)
    + 2 \ell_K(w)$}; \\
\mathbb C & \text{ if $p = \tfrac 1 2 \dim(\mathfrak g/\mathfrak k)
  -  \ell(wu) + 2 \ell_K(w)$}.
\end{cases}
\end{align*}
This is exactly Schmid's theorem (Theorem \ref{schmid}).

Therefore, this calculation shows that each module $J(w,\lambda)$ in the
Trauber resolution contributes exactly one cohomology class in $\mathfrak
n$-homology of $\Gamma(X,\mathcal I(Q,\tau))$.

\section{Character formulas}
\label{character_formulas}
In this section we calculate the characters of discrete series
representations $\Gamma(X, \mathcal I(Q,\tau))$ on the elliptic set.
We start with the trivial (and well-known) example of a compact Lie
group $G_0$. In this case, our result is just the well-known Weyl
character formula.

\subsection{The case of compact groups}
Let $G_0$ be a connected compact semisimple Lie group with
complexified Lie algebra $\mathfrak g$. We fix a maximal tours $T_0$
in $G_0$.  Let $\mathfrak t \subset \mathfrak g$ be its complexified
Lie algebra. Fix a set of positive roots $R^+$ in the root system $R$
of $(\mathfrak g,\mathfrak t)$. Then, $\mathfrak t$ and the root
subspaces of $\mathfrak g_\alpha$ corresponding to roots $\alpha \in
R^+$ span a Borel subalgebra $\mathfrak b$. We put $\mathfrak n =
[\mathfrak b,\mathfrak b]$.

Let $F$ be the irreducible finite-dimensional representation of $G_0$
with lowest weight $\lambda \in \mathfrak h^*$. We want to determine
the character $\ch(F)$ of $F$. Since all conjugacy classes in $G_0$
intersect $T_0$, it is enough to give a formula for $\ch(F)$ on
$T_0$.

Clearly, $F$ and the standard complex $C^\cdot(\mathfrak n,F)$, with
$C^{-p} (\mathfrak n,F) = \bigwedge^p \mathfrak n \otimes F$, $p \in
\mathbb Z$, have natural structures of finite-dimensional
representations of $T$. Therefore, their characters $\ch_T$ are
well-defined. By Euler's principle, we have
$$
\sum_{p = 0}^n (-1)^p \ch_T(C^{-p}(\mathfrak n,F))
= \sum_{p = 0}^n (-1)^p \ch_T(H_p(\mathfrak n,F))
$$
where $n = \dim \mathfrak n$.

Moreover, we have
$$
\ch_T(C^{-p}(\mathfrak n,F)) = \ch_T\left(\bigwedge^p \mathfrak n \otimes
F\right) =
\ch_T\left(\bigwedge^p \mathfrak n\right) \cdot \ch(F)
$$
on $T_0$.  Hence, it follows that
\begin{align*}
\sum_{p = 0}^n (-1)^p \ch_T(H_p(\mathfrak n,F))
&= \sum_{p = 0}^n (-1)^p \ch_T(C^{-p}(\mathfrak n,F))\\
& =
\ch(F) \cdot
\sum_{p = 0}^n (-1)^p \ch_T\left(\bigwedge^p \mathfrak n\right).
\end{align*}
This finally implies the following formula
$$
\ch(F) = \frac { \sum_{p=0}^n (-1)^p \ch_T(H_p(\mathfrak n,F))}
  {\sum_{p=0}^n (-1)^p \ch_T(\bigwedge^p \mathfrak n)} 
$$ on $T_0$.
  
Let $e^\mu : T \longrightarrow \mathbb C^*$ be the morphism with the
differential corresponding under specialization to $\mu \in \mathfrak
h^*$. Then the weights in $\bigwedge^p \mathfrak n$ are the sums of
all sets of $p$ different roots from $R^+$; i.e, we have
$$
\ch_T \left(\bigwedge^p \mathfrak n\right)
= \sum_{P \subset R^+, \ \Card(P) = p} \left(\prod_{\alpha \in P} e^\alpha \right).
$$
Therefore, we have
\begin{align*}
\sum_{p=0}^n (-1)^p \ch_T\left(\bigwedge^p \mathfrak n \right)
&= \sum_{p=0}^n (-1)^p \left(\sum_{P \subset R^+, \ \Card(P) = p}
\left(\prod_{\alpha \in P} e^\alpha \right)\right)\\
&= \sum_{p=0}^n \left(\sum_{P \in R^+, \ \Card(P) = p}
\left(\prod_{\alpha \in P} (- e^\alpha)\right) \right) \\
&= \sum_{P \subset R^+} \left(\prod_{\alpha \in P} (- e^\alpha)\right)
= \prod_{\alpha \in R^+} (1 - e^\alpha).
\end{align*}

Hence, we have
$$
\ch(F) = \frac { \sum_{p=0}^n (-1)^p
\ch_T(H_p(\mathfrak n,F_\lambda))}
  {\prod_{\alpha \in R^+} (1 - e^\alpha)} 
  $$
  on $T_0$.  Using Kostant's theorem \ref{kostant}, we get
  $$
\ch_T(H_p(\mathfrak n,F))= \sum_{w \in W(p)} e^{w(\lambda-\rho)+\rho} 
  $$
and finally
$$
\sum_{p=0}^n (-1)^p \ch_T(H_p(\mathfrak n,F)) =
 \sum_{w \in W} (-1)^{\ell(w)} e^{w(\lambda-\rho)+\rho}.
 $$
 Putting everything together, we conclude that
 $$
 \ch(F) = \frac { \sum_{w \in W} (-1)^{\ell(w)} e^{w(\lambda -\rho) + \rho}}
 {\prod_{\alpha \in R^+} (1 - e^\alpha) }.
 $$
  This finally implies the following result.

  \begin{thm}[Weyl]
    Let $F$ be the irreducible finite-dimensional representation of
    $\mathfrak g$ with lowest weight $\lambda$. Then
    $$
\ch(F) = \frac { \sum_{w \in W} (-1)^{\ell(w)} e^{w(\lambda -\rho) + \rho}}
   {\prod_{\alpha \in R^+} (1 - e^\alpha) }
   $$
 on the maximal torus $T_0$ in $G_0$.
  \end{thm}
  
  \subsection{Discrete series}
  The formal Euler characteristic argument in the case of compact
  $G_0$ clearly does not work in the case of the discrete series of a
  noncompact $G_0$ since they are infinite-dimensional. It is replaced
  by an argument based on the Osborne formula
  \cite{osborne}.\footnote{Actually, we need only a special case for
  compact Cartan subgroups \cite[7.27]{osborne}, which is already
  implicit in \cite{l2coho}.}

 Let $G_0$ be a connected semisimple Lie group and $K_0$ its maximal
 compact subgroup. Let $C_0^\infty(G_0)$ be the space of
 complex-valued smooth functions on $G_0$ with compact support,
 equipped with the usual topology. The continuous dual of that space
 is the space of distributions on $G_0$. Let $V$ be a Harish-Chandra
 module of finite length, then the Harish-Chandra character $\ch(V)$
 of $V$ is a distribution on $G_0$
 \cite{hc_char}.\footnote{Harish-Chandra defines the character for a
 representation of $G_0$ on a Hilbert space. By \cite[Proposition
   8.23]{cm}, every Harish-Chandra module $V$ of finite length is the
 module of $K_0$-finite vectors of a subrepresentation of a
 representation of $G_0$ induced by a finite dimensional
 representation of a minimal parabolic subgroup $P_0$ of
 $G_0$. Therefore, Harish-Chandra's construction applies to the closure
 of $V$. Moreover, the character is the sum of $K_0$-finite matrix
 coefficients which are completely determined by the Harish-Chandra
 module \cite[Thm. 8.7]{cm}.}

 The characters are invariant under inner automorphisms of $G_0$,
 therefore they are {\em invariant} distributions on $G_0$. The center
 $\mathcal Z(\mathfrak g)$ of $\mathcal U(\mathfrak g)$ is naturally
 identified with the invariant differential operators on $G_0$. If $V$
 is an irreducible Harish-Chandra module in $\mathcal M(\mathcal
 U_\theta)$, its character $\ch(V)$ is annihilated by the maximal
 ideal $J_\theta$ in $\mathcal Z(\mathfrak g)$; i.e., it is an {\em
   invariant eigendistribution} on $G_0$. We denote the space of all
 invariant eigendistributions annihilated by $J_\theta$ by $\mathcal
 E(G_0, \theta)$. Harish-Chandra has shown that any distribution $T$ in
 $\mathcal E(G_0,\theta)$ is given by
$$
C_0^\infty (G_0) \ni f \longmapsto T(f) = \int_{G_0} f(g) \Theta_T(g) \, d \mu(g)
$$
where $\mu$ is a fixed Haar measure on $G_0$ and $\Theta_T$ is a
locally integrable real analytic function, constant on conjugacy
classes, on the set $G_0'$ of all regular elements in $G_0$
\cite[Thm.~2]{hc_reg_char}.

Therefore, for any irreducible Harish-Chandra module $V$ the character
$\ch(V)$ is given by
$$
C_0^\infty (G_0) \ni f \longmapsto \ch(V)(f) = \int_{G_0} f(g) \Theta_V(g) \,
d \mu(g)
$$
for a locally integrable real analytic function $\Theta_V$ on
$G_0'$ constant on conjugacy classes.

From now on, we assume that ranks of $G_0$ and $K_0$ are equal.  Let
$T_0$ be a maximal torus in $K_0$. An element $g \in G_0$ is {\em
  elliptic} if its adjoint action $\Ad(g)$ on $\mathfrak g$ is
semisimple and its eigenvalues are complex numbers of absolute value
$1$. Denote by $E$ the set of all regular elliptic elements in $G_0$.
Also denote by $T_0'$ the set of regular elements in $T_0$. Clearly,
$E$ is an open set in $G_0$, invariant under conjugation by elements
of $G_0$, and every conjugacy class in $E$ intersects $T_0$.

The restriction to $E$ of the function $\Theta_V$ for an irreducible
Harish-Chandra module is a real analytic function on $E$ constant on
conjugacy classes. The Osborne formula describes this function
explicitly. Let $\mathfrak b$ be a Borel subalgebra in $X$ such that
$\mathfrak t \subset \mathfrak b$ and $\mathfrak n = [\mathfrak b,
  \mathfrak b]$. Then, we have
$$
\Theta_V|_{T_0'} = \frac { \sum_{p=0}^n (-1)^p \ch_T(H_p(\mathfrak n,V))}
  {\prod_{\alpha \in R^+} (1 - e^\alpha)}
  $$
  on the regular elements in a compact Cartan subgroup $T_0$ of
  $G_0$. This is the special case of the Osborne formula we alluded to
  above. It is a generalization of the formula we used in calculations
  for compact group $G_0$ in the preceding section.

Now, we are going to apply this formula to the case of the discrete series
representation $V = \Gamma(X,\mathcal I(Q,\tau))$ as discussed in
Section \ref{nhomology}. Assume that $\mathfrak b$ is a Borel
subalgebra corresponding to a point in $Q$.
  
Let $\epsilon : W \longrightarrow \{\pm 1\}$ be the character of the
Weyl group $W$ given by $\epsilon(w) = \det (w)$, $w \in W$. Then,
$\epsilon(s) = -1$ for any reflection $s$ in $W$. Therefore, we have
$\epsilon(w) = (-1)^{\ell(w)}$ for any $w \in W$. In addition, since
$\epsilon(s) = -1$ for any reflection $s$ with respect to a compact
root, the restriction of $\epsilon$ to $W_K$ is equal to the
corresponding character of Weyl group $W_K$. 
First, by Theorem \ref{schmid}, the numerator in
the Osborne formula is equal to
\begin{align*}
  \sum_{p \in \mathbb Z_+} (-1)^p \ch_T(H_p(\mathfrak n, & \Gamma(X,
  \mathcal I(Q,\tau))) \\
& = \sum_{p \in \mathbb Z_+} (-1)^p \Big(\sum_{w \in W_K}
\ch_T (H_p(\mathfrak n,\Gamma(X, \mathcal I(Q,\tau)))_{(w\lambda + \rho)})\Big)\\
& =\sum_{w \in W_K} \Big(\sum_{p \in \mathbb Z_+} (-1)^p
\ch_T (H_p(\mathfrak n,\Gamma(X, \mathcal I(Q,\tau)))_{(w\lambda + \rho)}\Big)\\
& = \sum_{w \in W_K} (-1)^{\frac 1 2 \dim(\mathfrak g /\mathfrak k) - \ell(w)
  +2 \ell_K(w)} e^{w\lambda+\rho}\\
& = (-1)^{\frac 1 2 \dim(\mathfrak g/\mathfrak k)} \sum_{w \in W_K} (-1)^{\ell(w)}
e^{w  \lambda + \rho}.
\end{align*}
This finally implies the following result.

\begin{thm}
  \label{disc_char}
On the regular elements of the compact Cartan subgroup $T_0$ the
character of the discrete series representation $V = \Gamma(X,\mathcal
I(Q,\tau))$ is given by
$$ \Theta_V|_{T_0'}
= (-1)^{\frac 1 2 \dim(\mathfrak
  g/\mathfrak k)} \frac { \sum_{w \in W_K} (-1)^{\ell_K(w)}
  e^{w \lambda + \rho}} {\prod_{\alpha \in R^+}(1 - e^\alpha)} \quad .
  $$
\end{thm} 

\subsection{Relation with Harish-Chandra parametrization}
In this section, we relate the information from \ref{disc_char} with
Harish-Chandra's results on discrete series characters.

In \cite{ds2}, Harish-Chandra introduced the Schwartz space $\mathcal
C(G_0)$ consisting of complex-valued rapidly decreasing smooth
functions on $G_0$. It contains $C_0^\infty(G_0)$ as a dense subspace.
A distribution on $G_0$ is {\em tempered} if it extends to a
continuous linear form on $\mathcal C(G_0)$. We denote by $\mathcal
E_{temp}(G_0,\theta)$ the subspace of $\mathcal E(G_0,\theta)$
consisting of tempered invariant eigendistributions.

If $G_0$ and $K_0$ have the same rank, the set of regular elliptic
elements $E$ is open in $G_0$ and one can consider the restriction map
from the space of invariant eigendistributions on $G_0$ to the space
of invariant eigendistributions on $E$.

Harish-Chandra proved the following result:\footnote{An equivalent
result was proven later by Atiyah and Schmid in \cite{atyiah-schmid}
by different methods.}

\begin{thm}
  \label{restriction_cc}
Let $\theta$ be a $W$-orbit in $\mathfrak h^*$ consisting of regular
real elements. Then the restriction map to $E$ is injective on
$\mathcal E_{temp}(G_0,\theta)$.
\end{thm}

Since the characters of discrete series are tempered invariant
eigendistributions \cite[Lem.~76]{ds2}, they are completely determined
by their restrictions to $E$. In particular, we have the following
version of \ref{disc_char}:

\begin{thm}
The character of the discrete series representation $V =
\Gamma(X,\mathcal I(Q,\tau))$ is the unique tempered invariant
eigendistribution $\ch(V)$ satisfying
$$
\Theta_V|_{T_0'}
= (-1)^{\frac 1 2 \dim(\mathfrak
  g/\mathfrak k)} \frac { \sum_{w \in W_K} (-1)^{\ell_K(w)}
  e^{w \lambda + \rho}} {\prod_{\alpha \in R^+}(1 - e^\alpha)} \quad .
  $$
\end{thm}

This is equivalent to the Harish-Chandra formula for the characters of
discrete series \cite[Theorem 16]{ds2}. This establishes a precise
connection between the geometric parametrization and Harish-Chandra's
parametrization of the discrete series.

Moreover, this proves the existence of tempered invariant
distributions $\Theta_\lambda$ in \cite[Thm.~3]{ds1} which is the central
result of that paper.

In addition, it implies that for any $W$-orbit $\theta$
in $\mathfrak h^*$ consisting of regular real elements, the space of
all tempered invariant eigendistributions on $G_0$ is spanned by
the characters of the discrete series which are in $\mathcal M(\mathcal
U_\theta)$.

\section{The Blattner Conjecture}
\label{blattner_conj}
As before, let $G_0$ be a connected semisimple Lie group with finite
center and $K_0$ its maximal compact subgroup. Also, we assume that
the ranks of $G_0$ and $K_0$ are equal, so $G_0$ admits discrete
series representations. The formula for the multiplicity of a
finite-dimensional irreducible representations of $K_0$ in a discrete
series representation of $G_0$ was conjectured by Robert Blattner. His
conjecture was proved by Hecht and Schmid in
\cite{blattner}.\footnote{Their proof assumes that the group $G_0$ is
linear. Our argument removes that condition. } We shall give a very
simple proof of that formula using the geometric realization of
discrete series.

\subsection{Filtration by normal degree}
First we recall a natural filtration of a $\mathcal D$-module direct image for
a closed immersion. Let $Y$ be a smooth algebraic variety and $Z$ a
closed smooth subvariety of $Y$.  Denote by $i : Z \longrightarrow Y$
the inclusion morphism of $Z$ into $Y$. Let $\mathcal T_Y$
(resp.~$\mathcal T_Z$) be the tangent sheaf to $Y$ (resp.~Z). Denote by
$i^*(\mathcal T_Y)$ the $\mathcal O$-module inverse image of $\mathcal
T_Y$.  Define the {\em normal sheaf} to $Z$ as $\mathcal N_{Z|Y} =
i^*(\mathcal T_Y)/\mathcal T_Z$. Moreover, we denote by $\omega_Y$
(resp.~$\omega_Z$) the invertible $\mathcal O$-module of sections top
degree differential forms on $Y$ (resp. $Z$). Also, we put
$\omega_{Z|Y} = i^*(\omega_Y) \otimes_{\mathcal O_Z} \omega_Z^{-1}$.

Let $\mathcal D$ be a twisted sheaf of differential operators on
$Y$. Let $\mathcal D^i$ be the corresponding twisted sheaf of
differential operators on $Z$ \cite{book}. Let $i_\bullet$ be the
sheaf direct image functor and $i_*$ the $\mathcal O$-module direct
image functor. Then the $\mathcal D$-module direct image functor
$$
i_+(\mathcal V) = i_\bullet(\mathcal V \otimes_{\mathcal D^i}
\mathcal D_{Z \rightarrow Y})
  $$
  is an exact functor from the category of right $\mathcal
  D^i$-modules into the category of right $\mathcal D$-modules. As
  explained in the Appendix to \cite{hmsw1}, one defines a natural
  filtration of $\mathcal D_{Z \rightarrow Y}$ by left $\mathcal D^i$- and
  $i^{-1} \mathcal O_Y$-modules. By tensoring, we get a natural
  $\mathcal O_Y$-module filtration of $i_+(\mathcal V)$ which we call
  the filtration by {\em normal degree}. The corresponding graded module
  is
  $$ \Gr i_+(\mathcal V) = i_\bullet(\mathcal V \otimes_{\mathcal O_Z}
  S(\mathcal N_{Z|Y})) = i_*(\mathcal V \otimes_{\mathcal O_Z}
  S(\mathcal N_{Z|Y}))
  $$
(here $S(\mathcal N_{Z|Y})$ is the symmetric algebra of $\mathcal N_{Z|Y}$).
  Since the opposite sheaf of rings of $\mathcal D$ is also a sheaf
  of twisted differential operators, we can easily adapt this
  construction to left $\mathcal D$-modules as explained in
  \cite{bebe2}.  Going to left modules contributes a twist by
  $\omega_{Z|Y}^{-1}$ in above formula and we get that
$$
  \Gr i_+(\mathcal V) = i_*(\mathcal V \otimes_{\mathcal O_Z} \omega_{Z|Y}^{-1}
  \otimes_{\mathcal O_Z} S(\mathcal N_{Z|Y})).
  $$  

\subsection{Proof of the Blattner formula}
Let $T_0$ be a Cartan subgroup of $K_0$. Denote, as before, by $K$
and $T$ the complexifications of $K_0$ and $T_0$. Denote by $\mathfrak
k$ and $\mathfrak t$ the Lie algebras of $K$ and $T$. 

Let $Q$ be a closed $K$-orbit in the flag variety $X$ of $\mathfrak
g$. Let $R$ be the root system of $(\mathfrak g,\mathfrak t)$ in
$\mathfrak t^*$. We fix a set of positive roots $R_c^+$ of the system
of compact roots $R_c$ in $R$. Then, by the discussion in Section
\ref{closed_orbits}, there exists a unique set of positive roots $R^+$
in $R$ such that $R_c^+ = R_c \cap R^+$ and the Borel subalgebra
$\mathfrak b$ spanned by $\mathfrak t$ and root subspaces $\mathfrak
g_\alpha$ corresponding to roots $\alpha \in R^+$ is in $Q$. Let
$\rho$ (resp.~$\rho_c$) be the half-sum of roots in $R^+$
(resp.~$R_c^+$).  Put $\rho_n = \rho - \rho_c$.

Let $x$ be a point in $Q$ corresponding to the Borel subalgebra
$\mathfrak b$. The tangent space $T_x(X)$ of $X$ at $x$ is identified
with $\mathfrak g/\mathfrak b$.  This isomorphism identifies the
tangent space $T_x(Q)$ to the orbit $Q$ at $x$ with the image of
$\mathfrak k$ in $\mathfrak g/\mathfrak b$ which is isomorphic to
$\mathfrak k/(\mathfrak k \cap \mathfrak b)$ as a representation of
$B_K$. This implies the following result.

\begin{lem}
  \label{blattner-weights}
  \begin{enumerate}
    \item[(i)]
The $K$-equivariant $\mathcal O_Q$-module $\omega_{Q|X}$ is isomorphic
to $\mathcal O_Q(2\rho_n)$.
\item[(ii)]
  The weights of the representation of $T$ on the
  geometric fiber of $\mathcal N_{Q|X}$ are negative noncompact roots.
\end{enumerate}
  \end{lem}

  Let $\tau$ be an irreducible $K$-homogeneous connection on $Q$
  compatible with $\lambda + \rho$. If we consider the
  normal degree filtration of the standard Harish-Chandra sheaf
  $\mathcal I(Q,\tau)$, we get the short exact sequence
  $$
  0 \rightarrow \F_{p-1} \mathcal I(Q,\tau) \rightarrow \F_p \mathcal I(Q,\tau)
  \rightarrow
  i_\bullet(\omega_{Q|X}^{-1} \otimes_{\mathcal O_Q} \tau \otimes_{\mathcal O_Q}
  S^p(\mathcal N_{Q|X})) \rightarrow 0.
  $$ The last sheaf is an $\mathcal O_X$-module and by the Leray
  spectral sequence
\begin{align*}
H^p(X, i_\bullet(\omega_{Q|X}^{-1} \otimes_{\mathcal O_Q} \tau
\otimes_{\mathcal O_Q} S^p(\mathcal N_{Q|X}))) & = H^p(X,
i_*(\omega_{Q|X}^{-1} \otimes_{\mathcal O_Q} \tau \otimes_{\mathcal O_Q}
S^p(\mathcal N_{Q|X})))\\
&= H^p(Q, \omega_{Q|X}^{-1} \otimes_{\mathcal O_Q} \tau \otimes_{\mathcal O_Q}
S^p(\mathcal N_{Q|X}))
\end{align*}
since $i$ is an affine morphism.

 The $K$-equivariant coherent $\mathcal O_Q$-module $\omega_{Q|X}^{-1}
 \otimes_{\mathcal O_Q} \tau \otimes_{\mathcal O_Q} S^p(\mathcal
 N_{Q|X})$ has finite-dimensional cohomology; i.e., $H^p(Q,
 \omega_{Q|X}^{-1} \otimes_{\mathcal O_Q} \tau \otimes_{\mathcal O_Q}
 S^p(\mathcal N_{Q|X}))$ are finite-dimensional algebraic
 representations of $K$ for any $p \in \mathbb Z_+$. By induction on
 $p$, from the long exact sequence of cohomology associated to the
 above short exact sequence, we conclude that $H^q(X,\F_p \mathcal
 I(Q,\tau))$ are finite-dimensional algebraic representations of
 $K$. Moreover, since $K$ is reductive, all representations in this
 long exact sequence are semisimple.
  
Let $\nu$ be a $R_c^+$-antidominant weight and $F_\nu$ the irreducible
finite-dimensional representation with lowest weight $\nu$. By
applying the functor $\Hom_K(F_\nu, - )$ we get again a long exact
sequence of finite-dimensional vector spaces. Using the Euler
principle, we conclude that
\begin{multline*}
  \sum_{q \in \mathbb Z} (-1)^q \dim \Hom_K(F_\nu, H^q(X,\F_p
  \mathcal I(Q,\tau))) \\ = \sum_{q \in \mathbb Z} (-1)^q \dim
  \Hom_K(F_\nu, H^q(X,\F_{p-1} \mathcal I(Q,\tau))) \\ + \sum_{q \in
    \mathbb Z} (-1)^q \dim \Hom_K(F_\nu, H^q(Q, \omega_{Q|X}^{-1} \otimes_{\mathcal O_Q}
  \tau \otimes_{\mathcal O_Q} S^p(\mathcal N_{Q|X}))).
\end{multline*}
Hence, by induction on $p$, we conclude that
\begin{multline*}
 \sum_{q \in \mathbb Z} (-1)^q \dim \Hom_K(F_\nu, H^q(X,\F_p
  \mathcal I(Q,\tau))) \\ \sum_{s=0}^p \sum_{q \in
    \mathbb Z} (-1)^q \dim \Hom_K(F_\nu, H^q(Q, \omega_{Q|X}^{-1} \otimes_{\mathcal O_Q}
  \tau \otimes_{\mathcal O_Q} S^p(\mathcal N_{Q|X}))).
\end{multline*}

The geometric fiber $T_x(S^p(\mathcal N_{Q|X}))$ is a algebraic
representation of $B_K$. Since $B_K$ is solvable, by Lie's theorem,
there exists a full flag of $B_K$-invariant subspaces in
$T_x(S^p(\mathcal N_{Q|X}))$. By Lemma \ref{blattner-weights}.(ii),
the weights of that representation are the sums of $p$ negative
noncompact roots.  The flag of $T_x(S^p(\mathcal N_{Q|X}))$ defines a
filtration of $S^p(\mathcal N_{Q|X})$ by $K$-equivariant $\mathcal
O_Q$-modules such that the graded $\mathcal O_Q$-module is the direct
sum of all $\mathcal O_Q( - \kappa)$ where $\kappa$ is a sum of $p$
positive momcompact roots. By tensoring with $\omega_{Q|X}^{-1}
\otimes_{\mathcal O_Q} \tau$ we get a filtration of $\omega_{Q|X}^{-1}
\otimes_{\mathcal O_Q} \tau \otimes_{\mathcal O_Q} S^p(\mathcal
N_{Q|X})$ by $K$-equivariant $\mathcal O_Q$-modules such that the
corresponding graded module is the direct sum of all $\omega_{Q|X}^{-1}
\otimes_{\mathcal O_Q} \tau \otimes_{\mathcal O_Q} \mathcal
O_Q(-\kappa)$ where $\kappa$ is a sum of $p$ positive noncompact
roots.

As above, using the Euler principle, by induction on the length of
this filtration, we conclude that
\begin{multline*}
\sum_{q \in \mathbb Z} (-1)^q \dim \Hom_K(F_\nu, H^q(Q, \omega_{Q|X}^{-1}
\otimes_{\mathcal O_Q} \tau \otimes_{\mathcal O_Q}
S^p(\mathcal N_{Q|X}))) \\
= \sum_\kappa \sum_{q \in \mathbb Z} (-1)^q \dim \Hom_K(F_\nu, H^q(Q,\omega_{Q|X}^{-1}
\otimes_{\mathcal O_Q} \tau \otimes_{\mathcal O_Q} \mathcal
O_Q(-\kappa)))
\end{multline*}
where the sum goes over all sums $\kappa$ of $p$ noncompact positive roots.

Clearly, we have $\tau = \mathcal O_Q(\lambda + \rho)$.  Moreover, by
Lemma \ref{blattner-weights}.(i), we have
  $$
  \omega_{Q|X}^{-1} \otimes_{\mathcal O_Q} \tau \otimes_{\mathcal O_Q}
  \mathcal O_Q(-\kappa) = \mathcal O_Q(\lambda + \rho_c - \rho_n - \kappa).
  $$
  So, we finally get
\begin{multline*}
\sum_{q \in \mathbb Z} (-1)^q \dim \Hom_K(F_\nu, H^q(Q, \omega_{Q|X}^{-1}
\otimes_{\mathcal O_Q} \tau \otimes_{\mathcal O_Q}
S^p(\mathcal N_{Q|X}))) \\
= \sum_\kappa \sum_{q \in \mathbb Z} (-1)^q \dim \Hom_K(F_\nu,
H^q(Q,\mathcal O_Q(\lambda + \rho_c - \rho_n -\kappa)))
\end{multline*}

The right hand side can be calculated using the Borel-Weil-Bott
theorem. The cohomology of $\mathcal O_Q(\lambda + \rho_c - \rho_n -
\kappa)$ vanishes if $\lambda - \rho_n - \kappa$ is not regular with
respect to $R_c$. If $\lambda - \rho_n - \kappa$ is regular with
respect to $R_c$, only one cohomology group $H^q(Q, \mathcal
O_Q(\lambda + \rho_c - \rho_n -\kappa))$, $q \in \mathbb Z$, can be
nonzero. More precisely, if $\mu$ is an antidominant weight with
respect to $R_c^+$, $H^q(Q, \mathcal O_Q(\lambda + \rho_c - \rho_n -
\kappa)) = F_\mu$ if and only if $\lambda - \rho_n - \kappa = w(\mu -
\rho_c)$ for $w \in W_K$ such that $\ell_K(w) = q$.  Therefore,
$\Hom_K(F_\nu, H^q(Q, \mathcal O_Q(\lambda + \rho_c - \rho_n -
\kappa))) = \mathbb C$ if and only if $\kappa = \lambda - \rho_n -
w(\nu - \rho_c)$ is a sum of $p$ positive noncompact roots and $w \in
W_K$ is such that $\ell_K(w) = q$.

Let $P$ be the function on $\mathfrak h^*$ defined in the following
way: For any $\mu$ we set $P(\mu)$ to be the number of ways one can
represent $\mu$ as a sum of positive noncompact roots. Also, set $P_p$
to be the function on $\mathfrak h^*$ defined in the following way:
For any $\mu$, $P_p(\mu)$ is the number of ways one can represent
$\mu$ as a sum of $p$ positive noncompact roots. Clearly, we have $P =
\sum_{p = 0}^\infty P_p$. 

Therefore, we have
\begin{multline*}
\sum_{q \in \mathbb Z} (-1)^q \dim \Hom_K(F_\nu, H^q(Q, \omega_{Q|X}^{-1}
\otimes_{\mathcal O_Q} \tau \otimes_{\mathcal O_Q}
S^p(\mathcal N_{Q|X}))) \\
= \sum_{w \in W_K} (-1)^{\ell_K(w)} P_p(\lambda - \rho_n
- w(\nu - \rho_c)).
\end{multline*}
By the above formula, it follows that
\begin{multline*}
\sum_{q \in \mathbb Z} (-1)^q \dim \Hom_K(F_\nu, H^q(X,\F_p \mathcal I(Q,\tau)))\\
=
 \sum_{s=0}^p \sum_{w \in W_K} (-1)^{\ell_K(w)} P_s(\lambda - \rho_n
- w(\nu - \rho_c)).
\end{multline*}
Since cohomology commutes with direct limits, by taking the limit
as $p \rightarrow \infty$, we get
\begin{multline*}
\sum_{q \in \mathbb Z} (-1)^q \dim \Hom_K(F_\nu, H^q(X,\mathcal I(Q,\tau)))\\  =
\sum_{w \in W_K} (-1)^{\ell_K(w)} P(\lambda - \rho_n
- w(\nu - \rho_c)).
\end{multline*}

Finally, since $\lambda$ is antidominant, higher cohomologies of the
standard Harish-Chandra sheaf $\mathcal I(Q,\tau)$ vanish and we see
that
$$
\dim \Hom_K(F_\nu, \Gamma(X,\mathcal I(Q,\tau))) =
\sum_{w \in W_K} (-1)^{\ell_K(w)} P(\lambda - \rho_n
- w(\nu - \rho_c)).
$$
This completes the proof of the following result.

\begin{thm}[Hecht-Schmid]
  Let $\nu$ be an antidominant weight for the root system $R_c$ and
  $F_\nu$ the finite-dimensional irreducible algebraic representation
  of $K$ with lowest weight $\nu$. Then the multiplicity of the
  representation $F_\nu$ in the restriction of the discrete series
  representation $V = \Gamma(X,\mathcal I(Q,\tau))$ to $K$ is equal to
  $$
\sum_{w \in W_K} (-1)^{\ell_K(w)} P(\lambda - \rho_n - w(\nu - \rho_c)).
  $$
\end{thm}

This is Blattner's formula.

\appendix
\section{The Cousin resolution}
\label{cousinres}
In this appendix, for the convenience of the reader, we prove some
well-known results about Cousin resolutions in $\mathcal D$-module
setting (for references, see \cite{rd}, \cite{zuck}).

\subsection{Stratifications}
Let $X$ be a smooth algebraic variety\footnote{Our varieties do not
 have to be connected.} over $\mathbb C$. A subvariety $Y$ of $X$ is
{\em affinely imbedded} if for any affine open set $U$ in $X$ the set
$U \cap Y$ is affine. Clearly, if $Y$ is an
affinely imbedded subvariety and $U$ an open set in $X$, then $Y \cap
U$ is affinely imbedded in $U$.

Closed subvarieties of $X$ are affinely imbedded in $X$. Also, affine
subvarieties are affinely imbedded in $X$. The canonical inclusions of
affinely imbedded subvarieties are affine morphisms.

Let
$$
X = F_0 \supset F_1 \supset F_2 \supset \dots \supset F_n \supset F_{n+1} 
= \emptyset
$$ 
be a decreasing filtration of $X$ by closed algebraic subvarieties
satisfying the following condition:
\begin{enumerate}
\item[(S)] $F_p - F_{p+1}$ is a nonempty smooth affinely imbedded subvariety 
of $X$ of dimension $\dim X - p$ for $0 \le p \le n$.
\end{enumerate}
Then we call $(F_p ; 0 \le p \le n+1)$ a {\em stratification} of $X$
of {\em length} $n$.  We call the smooth affinely imbedded subvariety
$S_p = Z_p - Z_{p+1}$ the $(\dim X - p)$-dimensional {\em stratum} of $X$.

Let $0 \le q \le n$ and $X_q = X-F_{q+1}$. Then $X_q$ is smooth
subvariety of $X$ and $(F_p - F_{q+1}; 0 \le p \le q+1)$ is a
stratification of $X_q$ of length $q$.

\subsection{Local cohomology}
Let $\mathcal D_X$ the sheaf of differential operators in $X$. Let $Y$
be a closed smooth subvariety of $X$. Let $U = X - Y$ be the
complement of $Y$.  Let $i : Y \longrightarrow X$ and $j : U
\longrightarrow X$ be the canonical inclusions.

Denote by $\mathcal M(\mathcal D_X)$, $\mathcal M(\mathcal
D_Y)$ and $\mathcal M(\mathcal D_U)$ the corresponding categories
of quasicoherent $\mathcal D$-modules, and by $D^b(\mathcal D_X)$,
$D^b(\mathcal D_Y)$ and $D^b(\mathcal D_U)$ the corresponding derived
categories. Then the direct image functor $i_+ :\mathcal
M(\mathcal D_Y) \longrightarrow \mathcal M(\mathcal D_X)$ is
exact and lifts to an exact functor $i_+ : D^b(\mathcal D_Y)
\longrightarrow D^b(\mathcal D_X)$. Also, the direct image functor
$j_+ : \mathcal M(\mathcal D_U) \longrightarrow \mathcal
M(\mathcal D_X)$ is the ordinary sheaf-theoretic direct image
functor $j_\bullet$, it is left exact and its derived functor is 
$R j_\bullet : D^b(\mathcal D_U) \longrightarrow D^b(\mathcal D_X)$.

Let $\varGamma_Y : \mathcal M(\mathcal D_X) \longrightarrow \mathcal
M(\mathcal D_X)$ be the functor which associates to each
quasicoherent $\mathcal D_X$-module $\mathcal F$ the subsheaf of all
local sections $\varGamma_Y(\mathcal F)$ with support in $Y$. Let $R
\varGamma_Y : D^b(\mathcal D_X) \longrightarrow D^b(\mathcal D_X)$ the
corresponding functor of local cohomology.

Let $\mathcal F^\cdot$ be an object in $D^b(\mathcal D_X)$. Then we
have the canonical distinguished triangle
$$
\xymatrix@R=60pt{
& Rj_\bullet(\mathcal F^\cdot|_U) \ar[dl]^{[1]} \\
R\varGamma_Y(\mathcal F^\cdot) \ar[rr] & & \mathcal F^\cdot \ar[ul]
} 
$$
by \cite[VI.8.3]{bdm}. In addition, we have
$$
R\varGamma_Y(\mathcal F^\cdot) = i_+(Ri^!(\mathcal F^\cdot)).
$$ 
If $m = \dim X - \dim Y$, we have
$$
R^p i^!(\mathcal O_X) =
\begin{cases}
0 & \text{ for } p \ne m; \\
\mathcal O_Y & \text{ for } p = m;
\end{cases}
$$
i.e.,
$$
R\varGamma_Y(D(\mathcal O_X)) = D(i_+(\mathcal O_Y))[-m].
$$
Specializing the above distinguished triangle for $\mathcal F^\cdot
= D(\mathcal O_X)$ we get the distinguished triangle
$$
\xymatrix@R=60pt{
& Rj_\bullet(D(\mathcal O_U)) \ar[dl]^{[1]} \\
D(i_+(\mathcal O_Y))[-m] \ar[rr] & & D(\mathcal O_X) \ar[ul]
} 
$$ 
Assume that $Y$ is of codimension $1$ in $X$; i.e., $m = 1$. Then,
applying the cohomology functor to this distinguished triangle, we
get the long exact sequence:
\begin{multline*}
\dots \longrightarrow 0 \longrightarrow \mathcal O_X 
\longrightarrow j_\bullet(\mathcal O_U) 
\longrightarrow i_+(\mathcal O_Y) \longrightarrow 0 \longrightarrow \dots \\
 \dots\longrightarrow 0 \longrightarrow R^p j_\bullet(\mathcal O_U) 
\longrightarrow 0 \longrightarrow \dots .
\end{multline*}
Therefore, we get the exact sequence
$$
0 \longrightarrow \mathcal O_X 
\longrightarrow j_\bullet(\mathcal O_U) 
\longrightarrow i_+(\mathcal O_Y) \longrightarrow 0
$$
and $R^p j_\bullet(\mathcal O_Y) = 0$ for $p \ge 1$. 

Analogously, if $m > 1$, we get that 
$$
R^p j_\bullet(\mathcal O_U) = 
\begin{cases}
\mathcal O_X & \text{ for } p = 0 ; \\
i_+(\mathcal O_Y) & \text{ for } p = m-1; \\
0 & \text{ otherwise} .
\end{cases}
$$
Hence, we proved the following result.

\begin{lem}
\label{loccoh}
Let $m = \dim X - \dim Y$. Then:
\begin{enumerate}
\item[(i)]
If $m = 1$, we have the exact sequence
$$
0 \longrightarrow \mathcal O_X 
\longrightarrow j_\bullet(\mathcal O_U) 
\longrightarrow i_+(\mathcal O_Y) \longrightarrow 0
$$
and $R^p j_\bullet(\mathcal O_U) = 0$ for $p > 0$.
\item[(ii)]
If $m > 1$, we have
$$
R^p j_\bullet(\mathcal O_U) = 
\begin{cases}
\mathcal O_X & \text{ for } p = 0 ; \\
i_+(\mathcal O_Y) & \text{ for } p = m-1; \\
0 & \text{ otherwise.} 
\end{cases}
$$
\end{enumerate}
\end{lem}

\subsection{The Cousin resolution}
Let $X$ be a smooth algebraic variety with a stratification $(F_p ; 0
\le p \le n+1)$. Denote by $i_p : S_p \longrightarrow X$, $0 \le p \le
n$, the canonical inclusions of strata into $X$.

\begin{thm}
\label{cousin}
There exists a canonical exact sequence
$$
0 \rightarrow \mathcal O_X \rightarrow i_{0,+}(\mathcal O_{S_0}) 
\rightarrow i_{1,+}(\mathcal O_{S_1})  \rightarrow \dots 
\rightarrow i_{n,+}(\mathcal O_{S_n}) \rightarrow 0.
$$
\end{thm}

Therefore, the complex
$$
\dots \rightarrow 0 \rightarrow i_{0,+}(\mathcal O_{S_0}) 
\rightarrow i_{1,+}(\mathcal O_{S_1})  \rightarrow \dots 
\rightarrow i_{n,+}(\mathcal O_{S_n}) \rightarrow 0 \rightarrow \dots
$$ 
is a resolution of $\mathcal O_X$ in $\mathcal M(\mathcal
D_X)$. This resolution is called the {\em the Cousin resolution} of
$\mathcal O_X$.

We prove the theorem by induction in $n$. If $n = 1$, we have $F_0 =
X$ and $F_1 = Y$ is a smooth closed subvariety in $X$. Moreover, $S_0
= X - Y = U$ and $S_1 = Y$.  Therefore, the Cousin resolution reduces
to Lemma \ref{loccoh}.(i).

Assume that $n > 1$ and that the statement holds for $n-1$. Let $U =
X_{n-1}$ and denote by $j : U \longrightarrow X$ the natural
inclusion. Then $(F_p - F_n; 0 \le p \le n)$ define a stratification
of $U$. Denote by $j_p : S_p \longrightarrow U$ the natural
inclusions. By the induction assumption, we have the exact sequence
$$
0 \rightarrow \mathcal O_U \rightarrow j_{0,+}(\mathcal O_{S_0}) 
\rightarrow j_{1,+}(\mathcal O_{S_1})  \rightarrow \dots 
\rightarrow j_{n-1,+}(\mathcal O_{S_{n-1}}) \rightarrow 0 \quad .
$$ 
Therefore, we have a quasiisomorphism of $D(\mathcal O_U)$ into the
complex
$$
\dots \rightarrow 0 \rightarrow j_{0,+}(\mathcal O_{S_0}) 
\rightarrow j_{1,+}(\mathcal O_{S_1})  \rightarrow \dots 
\rightarrow j_{n-1,+}(\mathcal O_{S_{n-1}}) \rightarrow 0 
\rightarrow \dots \quad .
$$ 
By our assumption $j_p : S_p \longrightarrow U$ and $i_p:
S_p \longrightarrow X$ are affine morphisms. Therefore, the functors
$j_{p,+} : \mathcal M(\mathcal D_{S_p}) \longrightarrow
\mathcal M(\mathcal D_U)$ and $i_{p,+} : \mathcal
M(\mathcal D_{S_p}) \longrightarrow \mathcal M(\mathcal
D_X)$ are exact. Since $Rj_\bullet \circ j_{p,+} = i_{p,+}$, it
follows that the modules $j_{p,+}(\mathcal V)$ are acyclic for
$j_\bullet$ for any $\mathcal D_{S_p}$-module $\mathcal V$. Therefore,
by acting on the above complex by $j_\bullet$, we get the complex
$$
\dots \rightarrow 0 \rightarrow i_{0,+}(\mathcal O_{S_0}) 
\rightarrow i_{1,+}(\mathcal O_{S_1})  \rightarrow \dots 
\rightarrow i_{n-1,+}(\mathcal O_{S_{n-1}}) \rightarrow 0 
\rightarrow \dots \quad
$$ 
which computes $Rj_\bullet(D(\mathcal O_U))$. Since $U = X - S_n$,
by Lemma \ref{loccoh}.(ii), the kernel of the morphism
$i_{0,+}(\mathcal O_{S_0}) \rightarrow i_{1,+}(\mathcal O_{S_1})$
has to be $\mathcal O_X$, and the cokernel of the morphism
$i_{n-2,+}(\mathcal O_{S_{n-2}}) \rightarrow
i_{n-1,+}(\mathcal O_{S_{n-1}})$ has to be $i_{n,+}(\mathcal
O_{S_n})$. This clearly establishes the induction step.

\subsection{Variants}
In the main text we use the ``twisted'' versions of Theorem \ref{cousin}.

First, let $\mathcal L$ be an invertible $\mathcal O_X$-module on $X$
and $\mathcal D_{\mathcal L}$ the sheaf of differential operators on
$\mathcal L$. By tensoring by $\mathcal L$, we get immediately the
following version of Theorem \ref{cousin}.

\begin{thm}
\label{twcousin}
There exists a canonical exact sequence
$$
0 \rightarrow \mathcal L \rightarrow i_{0,+}(i_0^*(\mathcal L)) 
\rightarrow i_{1,+}(i_1^*(\mathcal L))  \rightarrow \dots 
\rightarrow i_{n,+}(i_n^*(\mathcal L)) \rightarrow 0
$$
in $\mathcal M(\mathcal D_{\mathcal L})$.
\end{thm}

Finally, let $Y$ be a smooth closed subvariety of $X$. Denote by $i :
Y \longrightarrow X$ the canonical inclusion and by $\mathcal D^i$ the
twisted sheaf of differential operators on $Y$ induced by the twisted
sheaf of differential operators $\mathcal D$ on
$X$ \cite[1.1]{book}. Let $\mathcal L$ be an invertible $\mathcal
O_Y$-module on $Y$. Assume that $\mathcal D$ is the twisted sheaf of
differential operators on $X$ such that $\mathcal D^i = \mathcal
D_{\mathcal L}$.

Let $(F_p ; 0 \le p \le n+1)$ be a stratification of $Y$. Denote by
$i_p : S_i \longrightarrow X$ and $k_p: S_p \longrightarrow Y$
the natural inclusions of strata $S_p$, $0 \le p \le n$, into $X$
(resp.~$Y$). Then, applying the Kashiwara equivalence of categories to
closed immersion $i$, from Theorem \ref{twcousin} we get the following result.

\begin{thm}
\label{itwcousin}
There exists a canonical exact sequence 
$$
0 \rightarrow i_+(\mathcal L) \rightarrow i_{0,+}(k_0^*(\mathcal L)) 
\rightarrow i_{1,+}(k_1^*(\mathcal L))  \rightarrow \dots 
\rightarrow i_{n,+}(k_n^*(\mathcal L)) \rightarrow 0
$$
in $\mathcal M(\mathcal D)$.
\end{thm}

\bibliography{discrete}
\bibliographystyle{amsplain}
\end{document}